\documentclass[11pt]{amsart}
  \usepackage[margin=3.3cm]{geometry}

\usepackage{amsfonts,amsmath,amssymb,color,esint}
\usepackage{enumerate}
\usepackage{graphicx}
\usepackage{tikz}
\usetikzlibrary{decorations.markings}
\usetikzlibrary{plotmarks}
\usetikzlibrary{patterns}
\numberwithin{equation}{section}

\usepackage{amsthm}
\usepackage[babel=true,kerning=true]{microtype}

\usepackage{amsthm,leqno}
\usepackage{hyperref}

\newtheorem{theo}{Theorem}[section]

\newtheorem{prop}{Proposition}[section]

\newcommand{\eps}{\varepsilon}

\newcommand{\R}{\mathbb{R}}

\begin{document}

\title[Laplace eigenvalues and minimal spheres into ellipsoids]{Laplace eigenvalues and non-planar minimal spheres into 3-dimensional ellipsoids}
\author{Romain Petrides}
\address{Romain Petrides, Universit\'e de Paris, Institut de Math\'ematiques de Jussieu - Paris Rive Gauche, b\^atiment Sophie Germain, 75205 PARIS Cedex 13, France}
\email{romain.petrides@imj-prg.fr}

\begin{abstract} 
We give a sufficient condition for branched minimal immersions of spheres into ellipsoids to be \textit{embedded}: we show that if the coordinate functions of the branched minimal immersion are first or second eigenfunctions with respect to a natural metric on the sphere involved in this problem, called \textit{critical metric}, then it is an embedding into a 2-dimensional ellipsoid or an immersion into a 3-dimensional ellipsoid. If in addition this critical metric is rotationally symmetric with respect to an axis and symmetric with respect to the orthogonal plane then the immersion is an embedding. We also give a construction of non-planar minimal spheres into 3-dimensional ellipsoids such that the coordinate functions are first and second eigenfunctions with respect to the critical metric obtained by maximization of linear combinations of first and second (area renormalized) Laplace eigenvalues among metrics on the sphere having the aforementionned symmetry constraints or not.
\end{abstract}

\maketitle

A problem raised by Yau in 1987 (\cite{Yau2}, section 4) asked for existence of non-planar embedded minimal spheres into 3-ellipsoids. By definition, a surface in $\mathbb{R}^4$ is planar if it lies into a hyperplane. This result is a priori surprising since for the analogous problem of geodesics, a classical result by Morse \cite{morse} stated that some 2-ellipsoids contain at most three embedded closed geodesics, the equatorial ones. However, in higher dimension, fruitful works initiated by Marques and Neves \cite{MarquesNeves} lead to the general resolution of the following famous Yau's conjecture by min-max methods (see Song \cite{song}): any manifold of dimension 3 admits infinitely many smoothly embedded minimal surfaces. This spectacular result shows that the problem of geodesics is not a good analogy to understand the embeddedness of minimal surfaces in higher dimension, but we do not get much information about the topology of these embedded minimal surfaces yet.

Therefore, we can ask to build many numbers of embedded minimal \textit{spheres} into ellipsoids. Of course, ellipsoids contain at least $4$ planar embedded spheres (equatorial spheres) but one might expect more ones, which are non-planar. This question, left open for years was recently solved by Haslhofer and Ketover \cite{hk}: sufficiently elongated ellipsoids contain an embedded non-planar minimal sphere. They combined Mean Curvature Flow and Min-Max methods to prove this result. Later, using Bifurcation Theory and the symmetries that arise if at least two semi-axes coincide, Bettiol and Piccione \cite{BP22} showed the existence of arbitrarily many distinct embedded non-planar minimal 2-spheres in sufficiently elongated ellipsoids of revolution.

In the current paper, we give a new method to build non-planar embedded minimal spheres and a new way to describe them. Indeed, we noticed in \cite{Pet20} (see also section \ref{minimmellispoid} in the current paper) that if $\Sigma$ is a minimal surface into some ellipsoid $\mathcal{E} \subset \mathbb{R}^{n}$ of parameters $p = \left(p_1,\cdots,p_n\right)$, with $p_i >0$, defined by
$$ \mathcal{E} = \{ (x_1,\cdots,x_n) \in \mathbb{R}^{n} ; p_1 x_1^2 + \cdots + p_n x_n^2 = 1 \} \hskip.1cm,$$
endowed with the induced metric of the Euclidean metric $\xi$, then the coordinate functions $(x_1,\cdots,x_n)$ of $\R^{n}$ are eigenfunctions with eigenvalues $(p_1,\cdots,p_n)$, with respect to the Laplacian of $(\Sigma,g)$ endowed with the metric 
\begin{equation} \label{defg} g = \frac{\left\vert H_\Sigma(x)\right\vert}{\left(p_1^2 x_1^2 + \cdots + p_n^2 x_n^2 \right)^{\frac{1}{2}}} \xi \end{equation}
where $H_\Sigma$ is the mean curvature vector of $\Sigma$ in $\mathbb{R}^n$. In fact, this simple remark gives a way to describe all the minimal immersions into any ellipsoids with the point of view of spectral geometry. This is a generalization of Takahashi's result \cite{takahashi} in the case of the sphere $p_1 = \cdots =p_n $, where coordinate functions are Laplace eigenfunctions with respect to the induced metric on $\Sigma$ of the Euclidean metric. 

One natural question is: given a minimal embedding into the ellipsoid $\mathcal{E}$, are the coordinate functions associated to the $n$ first eigenvalues counted with multiplicity of the Laplacian on $\Sigma$ ? In general, such a question is very hard to answer. If we just assume that $\mathcal{E}$ is a 3-sphere ($n=3$ and $p_1 = p_2 = p_3$) and that $\Sigma$ is an embedded surface, Yau conjectured that the coordinate functions are first eigenfunctions. If $\Sigma$ is topologically a sphere, or a torus, this conjecture holds true because there are uniqueness results: the equatorial spheres are the unique embedded minimal spheres by \cite{almgren}, and the Clifford torus is the unique embedded minimal torus \cite{brendle} \cite{hks}. The uniqueness result for the torus was known as the Lawson conjecture, left open for years. Notice that this conjecture by Yau is much more general since it implies the Lawson conjecture. Indeed, the unique immersed tori into some sphere by first eigenfunctions are the Clifford torus (in $\mathbb{S}^3$) and the flat equilateral torus (in $\mathbb{S}^5$) \cite{Ilias_ElSoufi}. In that spirit, it would be interesting to know the spectral indices of the parameters which appear in the previous constructions of non-planar embedded ellipsoids. Counting the nodal domains of coordinates of the arbitrarily many embedded non-planar minimal 2-spheres into ellipsoids obtained in \cite{BP22}  (of parameters ($p_1,p_2,p_3,p_3)$ and we assume in addition $p_2=p_3$), we remark that only one can be associated to the 4 first eigenvalues counted with multiplicity of the Laplacian associated to the metric defined in \eqref{defg}.

For a first study of this spectral point of view on minimal surfaces into ellipsoids, it is already interesting to assume $\Sigma = \mathbb{S}^2$ and to look for non-planar minimal spheres into rotationally symmetric ellipsoids having parameters which are first and second eigenvalues. Indeed, the aforementioned spectral characterization of minimal immersions into ellipsoids is more precise: the metric $g$ defined on $\Sigma$ by \eqref{defg} can be viewed as a critical metric on functionals defined as combinations of some eigenvalues, on the set of metrics of $\Sigma$. This characterization was observed in \cite{Pet20} (see also \cite{PT23}). It was a generalization of a result by Nadirashvili \cite{nadirashvili} for one eigenvalue.

Therefore, we perform a variational method on some combinations of first and second eigenvalues to build the expected minimal spheres. Variational methods if only one eigenvalue appears were intensively studied in the past decade after the seminal works by Nadirashvili \cite{nadirashvili} for the first eigenvalue on tori and Fraser and Schoen \cite{fs} for the first Steklov eigenvalue on zero genus surfaces with boundary (see also \cite{petrides} \cite{petrides-2} \cite{Pet19} \cite{Pet22}). The following result is a first complete application of a variational method for combination of different eigenvalues, started in \cite{Pet20}:
\begin{theo} \label{theomain} There is a one parameter family $(p_t,q_t)_{t > t_{\star}}$ for some $0 < t_{\star}<1$ such that there is an embedded non planar minimal topological sphere $S_t$ into the rotationally symmetric ellipsoid 
$$\mathcal{E}_t := \{ x\in \mathbb{R}^4 ; p_t x_0^2 + q_t x_1^2 + q_t x_2^2 + q_t x_3^2 = 1 \}$$ 
and such that the coordinates $x_0 , x_2 , x_2 , x_3$ are first and second Laplace eigenfunctions for the natural induced metric 
$$ g_t = \frac{H_{S_t}(x)}{\left(p_t^2 x_0^2 + q_t^2 x_1^2 + q_t^2 x_2^2 + q_t^2 x_3^2 \right)^{\frac{1}{2}}} \xi$$
on the sphere $S_t$, where $\xi$ is the standard Euclidean metric and $H_{S_t}$ is the mean curvature of $S_t$. For any $t_1 < t_2$, $S_{t_1}$ is not isometric to $S_{t_1}$. Moreover $t \mapsto p_t$ is decreasing and $t\mapsto q_t$ is increasing, and we have that $p_t \to 0$ and $q_t \to 16 \pi$ as $t\to +\infty$ and that $S_t$ converges as varifolds to the round sphere $\{0\} \times \mathbb{S}^2$ with multiplicity $2$ as $t\to +\infty$.
\end{theo}
We refer to section \ref{sectionpresentation} for a detailed presentation of the construction of this family. We would like to emphasize that the surfaces built by this method in Theorem \ref{theomain} are embedded. Indeed, variational methods on eigenvalues only provide us immersed and possibly branched minimal surfaces, but here, we have the special case of spheres immersed by first and second eigenvalues:

\begin{theo} \label{theoembedded} If $\Phi : \mathbb{S}^2 \mapsto \mathcal{E}$ is a minimal branched conformal immersion into some ellipsoid of parameter $p = \left(p_1,\cdots,p_n \right)$ such that for any $i\in \{1,\cdots, n\}$, $p_i$ is a first or second eigenvalue of $(\mathbb{S}^2,g)$ (and we have $ \{p,q\} = \{p_1,\cdots,p_n\} $ with $p\leq q$), where 
$$ g = \frac{\left\vert H_S\circ \Phi  \right\vert}{\left(p_1^2 \Phi_1^2 + \cdots + p_n^2 \Phi_n^2 \right)^{\frac{1}{2}}} \Phi_{\star}\xi. = e^{2v} dA_{\mathbb{S}^2} $$
Then up to rearrangements of coordinates, $S := \Phi\left( \mathbb{S}^2 \right) \subset \R^3 \times \{0\}$ and $\Phi$ is embedded or $S \subset \R^4 \times \{0\}$ and $\Phi$ does not have any branched point.
If in addition $e^{2v}$ is rotationally symmetric with respect to the axis $x_3$ and symmetric with respect to $\{x_3 = 0 \}$, then $\Phi$ is an embedding.
\end{theo}
Theorem \ref{theoembedded} is similar to the result of embeddedness of (possibly branched) minimal free boundary immersions into balls \textit{by first Steklov eigenfunctions} for surfaces with boundary of \textit{genus $0$} \cite{fs} but the proof of Theorem \ref{theoembedded} needs more refinement (see section \ref{sectionembeddedsphere}). We do not know if it is possible to remove the symmetry assumptions on $e^{2v}$ and get the same conclusion, but we also give a general picture and slightly weaker natural assumptions on the road to prove the embeddedness. In particular, as a first step, we use that the coordinates are independent first or second eigenfunctions, and that their multiplicity is bounded by $3$ (see \cite{HON}) to deduce the maximal dimension of the target manifold. Notice also that Theorem \ref{theomain} gives examples of metrics on the sphere such that the multiplicity $3$ for the second eigenvalue is realized. This question of optimality of $3$ was left open in \cite{HON}.

\medskip

The main idea to prove Theorem \ref{theomain} is to look for Riemannian metrics $g_t$ that maximize linear combinations of first and second eigenvalues $g\mapsto \bar{\lambda}_1(g) + t \bar{\lambda}_2(g)$ among all the metrics on $\mathbb{S}^2$. Here we define for any metric $g$, 
$$\bar{\lambda}_k(g) = \lambda_k(g) A_g(\mathbb{S}^2),$$
the scale-invariant $k$-th eigenvalue where $A_g(\mathbb{S}^2)$ is the area of the sphere with respect to $g$ and
$$ \lambda_0(g) = 0 < \lambda_1(g) \leq \lambda_2(g) \leq \cdots \leq \lambda_k(g) \to +\infty \text{ as } k\to +\infty $$
are the eigenvalues of the Laplacian $\Delta_g$ with respect to $g$ on $\mathbb{S}^2$ counted with multiplicity. The parameter $t > 0$ is exactly the one that appears in Theorem \ref{theomain}. 

\medskip

Let's recall some results about maximization/minimization problems of Laplace eigenvalues on a surface $\Sigma$. It is well known from works by Yang-Yau \cite{yangyau} and Li-Yau \cite{liyau} (first eigenvalues) and Korevaar \cite{korevaar} (higher eigenvalues) that $g \mapsto \overline{\lambda}_k(g) $ is bounded by a constant depending only on $k$ and the topology of $\Sigma$. We know that $\inf_g \overline{\lambda}_k(g) = 0$ and that it is never realized on connected surfaces. Indeed, minimizing sequences are the classical Cheeger dumbells. The simplest way to hope for critical metrics is then a maximization. Therefore, we ask the following very natural questions: can we compute the optimal constant 
$$\Lambda_k(\Sigma) = \sup_g \overline{\lambda}_k(g)$$ 
with respect to the topology and $k$ ? Is there a metric which realizes this supremum ?

Notice that Yang-Yau's and Li-Yau's inequalites were a refinement of a celebrated result by Hersch on the sphere to any topology. Since our paper is devoted to the spherical case for the bottom of the spectrum, let's highlight the Hersch inequality:
\begin{theo}[Hersch \cite{hersch}] Let $g$ be a smooth metric on $\mathbb{S}^2$, then
\begin{equation} \label{eqHersch} \frac{1}{\bar{\lambda}_1(g)} + \frac{1}{\bar{\lambda}_2(g)} + \frac{1}{\bar{\lambda}_3(g)} \geq \frac{3}{8\pi} \end{equation}
with equality if and only if $g$ is a round metric.
\end{theo}
From this theorem, he deduced that for metrics $g$ on $\mathbb{S}^2$ again,
\begin{equation} \label{eqHersch2} \frac{1}{\bar{\lambda}_1(g)} + \frac{2}{\bar{\lambda}_2(g)} \geq \frac{3}{8\pi} \end{equation}
with equality if and only if $g$ is a round metric and that
\begin{equation} \label{eqHersch3} \bar{\lambda}_1(g) \leq 8\pi \end{equation}
with equality if and only if $g$ is a round metric. The latter inequality gives a complete answer to the two previous questions on the sphere. However, the refinements by Yang and Yau of the inequality \eqref{eqHersch3} are not optimal anymore for surfaces with higher genuses, except in genus $2$ (see \cite{karpukhin-2}), and they left open the question of existence of maximizers for years. In the case of the tori, Berger \cite{berger} conjectured that the flat equilateral torus is the maximizer of the first eigenvalue. This result was proved by Nadirashvili \cite{nadirashvili}, opening knew methods to build maximizers of one eigenvalue for any topology, notably using that the maximizers have to be induced metrics of a minimal immersion into a sphere.

After several works for (not necessarily oriented) surfaces with low genuses and for $k\leq 2$ initiated by \cite{nadirashvili}, a global variational procedure for maximization of one eigenvalue was then given in \cite{petrides} for $k=1$ and in \cite{petrides-2} for $k>1$. We obtain new maximizers of higher genuses (see \cite{petrides}) and the combination of \cite{petrides-2} and \cite{Ejiri} leeds to an explicit computation of $\Lambda_k(\mathbb{S}^2)$ (see \cite{knpp}) and this method was extended for $\Lambda_k(\mathbb{RP}^2)$ \cite{karpukhin-3}. Since then, other different variational methods also appeared (see e.g \cite{KS} in the special cases $k=1$ and $k=2$). 

We highlight that while the round sphere is the unique maximizer of $g \mapsto \bar{\lambda}_1(g)$, the functional $g \mapsto \bar{\lambda}_2(g)$ does not have any maximizer at all. Indeed, it is known from \cite{nadirashvili-2}, \cite{petrides} and \cite{knpp} that 
\begin{equation} \label{eqk=2sphere} \Lambda_2(\mathbb{S}^2) = 16\pi \text{ and that } \bar{\lambda}_2(g) < 16\pi \end{equation}
for any smooth metric $g$ on the sphere. $16\pi$ corresponds to the renormalized eigenvalue of two spheres of same area. In fact, this property is general for higher eigenvalues (see \cite{knpp}), $\Lambda_k(\mathbb{S}^2) = 8 \pi k$ and there is not any maximizer but we can build maximizing sequences converging to a disjoint union of $k$ spheres with same area.

\medskip

In the current paper, we use the very general result coming from \cite{Pet20} (see also \cite{Pet22}) of the case of one eigenvalue, for variational problems on \textit{combination} of eigenvalues. For instance, the functional involved in the left-hand side of the Hersch inequality \eqref{eqHersch} is admissible in this setting: we ask for minimizers. Since the motivation of the current paper is to build embedded minimal spheres into ellipsoids, we focus on combinations of two eigenvalues because of Theorem \eqref{theoembedded}, so that we look at functionals like the one involved in \eqref{eqHersch2}. Of course, the inequality \eqref{eqHersch2} shows that we do not build non-planar minimal spheres by minimization of this functional. However, the admissible functionals given in \cite{Pet20} (see Theorem \ref{theomin} below) are numerous and we can choose for instance functionals like
\begin{equation} \label{eqfunctionalone} g \mapsto \frac{1}{\bar{\lambda}_1(g)} + \frac{s}{\bar{\lambda}_2(g)} \end{equation}
for some parameter $s>0$. For $s\leq 2$, the round sphere is again the unique minimizer but we can hope that there are some $s>2$ such that minimizers give existence of non-planar minimal spheres into ellipsoids. As explained in Section \ref{sectionpresentation}, there is such a parameter $s$. In order to prove Theorem \ref{theomain}, we figured out that minimizing the functional
\begin{equation} \label{eqfunctionaltwo} g \mapsto \frac{1}{\bar{\lambda}_1(g)+ t \bar{\lambda}_2(g)} \end{equation}
for some parameter $t>0$ is also interesting to build a one parameter family of non planar minimal surfaces into ellipsoids (see Section \ref{sectionpresentation}). However, for deep reasons (see Theorem \ref{theomin} below), we need to check some "gap properties" on eigenvalues for this one. 

More precisely, we need the sequence of eigenvalues $\bar{\lambda}_1(g_k)$ associated to minimizing sequence of metrics $g_k$ built in \cite{Pet20} for the functional \eqref{eqfunctionaltwo} not to converge to $0$. Indeed, if $\bar{\lambda}_1(g_k)\to 0$, we cannot prevent from topological degeneration of the sequence of the surfaces $(\mathbb{S}^2,g_k)$ as $k\to +\infty$: one can loose connectedness of the limiting surface. This degeneration occurs if we look at the functionnal $g \mapsto \bar{\lambda}_2(g)$. Indeed, as mentioned in \eqref{eqk=2sphere}, $g \mapsto \bar{\lambda}_2(g)$ cannot have maximizers but $\Lambda_2(\mathbb{S}^2) = 16\pi$ corresponds to the second eigenvalue of the disjoint union of two round spheres with same area. Then natural maximizing sequences are sequences of metrics on the sphere which "bubble converge" to two round spheres with same area. However, using a very precise asymptotic analysis on the first and second eigenvalues of such a sequence of metrics, we prove that
\begin{theo} \label{theoasymptotic}
There is a one parameter family of metrics $g_{\eps}$ on $\mathbb{S}^2$ such that
$$ \bar{\lambda}_1(g_{\eps}) = \frac{4\pi}{\ln\frac{1}{\eps}} + O\left( \frac{1}{\left(\ln\frac{1}{\eps}\right)^2}\right) \text{ and } \bar{\lambda}_2(g_{\eps}) = 16\pi - 48\pi \eps^2 + o\left( \eps^2\right)$$
as $\eps \to 0$.
\end{theo}
If for one minimizing sequence $g_k$ for \eqref{eqfunctionaltwo}, we have $\bar{\lambda}_1(g_k) \to 0$, then we must have $sup_g\left( \bar{\lambda}_1(g)+ t \bar{\lambda}_2(g)\right) = 16\pi t $ by \eqref{eqk=2sphere}. Testing this functional for the one parameter family $g_\eps$ given by Theorem \ref{theoasymptotic}, we have
$$  \bar{\lambda}_1(g_{\eps})+ t \bar{\lambda}_2(g_{\eps}) = 16\pi t + \frac{4\pi}{\ln\frac{1}{\eps}} + O\left( \frac{1}{\left(\ln\frac{1}{\eps}\right)^2}\right) \text{ as } \eps \to 0 $$
so that $sup_g \left( \bar{\lambda}_1(g)+ t \bar{\lambda}_2(g) \right) > 16\pi t $. Therefore, we obtain the "gap assumption" that garanties convergence of some minimizing sequences to a smooth metric on $\mathbb{S}^2$. 

\medskip

As a final remark, we refer to \cite{Pet23} to read in another way the methods we perform in the current paper. One can observe fruitful similarities and subtle differences between the closed case (Laplace eigenvalues) and the case with boundary (Steklov eigenvalues).

\medskip

The paper is organised as follows:

In Section \ref{sectionpresentation}, we give a general presentation of the proof of Theorem \ref{theomain}. In particular, we emphasize that the optimization of combinations of eigenvalues is given in a space of \textit{symmetric} conformal factors. We will refer to \cite{Pet22} and \cite{Pet23} for detailed steps of the method. In Section \ref{sectioncriticalmetrics}, we describe the link between critical metrics of combinations of eigenvalues and minimal immersions into ellipsoids. We then focus on 2-dimensional ellipsoids associated to first and second eigenvalues in order to give conditions such that they are not critical or not minimizers for choices of combination of these eigenvalues. Section \ref{sectionembeddedsphere} is devoted to the proof of Theorem \ref{theoembedded} and Section \ref{sectionasymptotic} is devoted to the proof of Theorem \ref{theoasymptotic}.

%%%%%%%%%%%%%%%%

\section{Precise presentation of the results} \label{sectionpresentation}

Let $F : \left(\mathbb{R^*}\right)^2 \to \mathbb{R}_+$ be a $\mathcal{C}^1$ function such that $\partial_1 F$ and $\partial_2 F$ are negative functions. The following first result characterizes the critical metrics for the functional $g\mapsto F\left(\bar{\lambda}_1(g), \bar{\lambda}_2(g) \right) $.

\begin{theo} \label{theocriticalembedded}
Let $g$ be a critical metric for $g\mapsto F\left(\bar{\lambda}_1(g), \bar{\lambda}_2(g) \right)$ on $\mathbb{S}^2$. Then 

Either $\lambda_1(g) = \lambda_2(g)$ and $g$ is a round metric, 

Or $\lambda_1(g) \neq \lambda_2(g)$ and there is a minimal isometric immersion $\Phi : \left(\mathbb{S}^2,g_0 \right) \to \mathcal{E}$ such that $g = e^{2u}g_0$ where $\mathcal{E}$ is one of the following ellipsoids
$$ \mathcal{ E }=  \mathcal{ E }_2 := \left\{ x = (x_0,x_1,x_2) \in \mathbb{R}^3 ; \bar{\lambda}_1(g) x_0^2 +  \bar{\lambda}_2(g) \left(x_1^2 + x_2^2 \right) = 1 \right\} $$
$$\text{ or }\mathcal{ E }= \mathcal{ E }_3 := \left\{ x = (x_0,x_1,x_2,x_3) \in \mathbb{R}^4 ; \bar{\lambda}_1(g) x_0^2 +  \bar{\lambda}_2(g) \left(x_1^2 + x_2^2 + x_3^2\right) = 1 \right\} $$
endowed with the induced metric on $\mathcal{E}$ of the Euclidean metric and the conformal factor $e^{2u}$ is defined as
$$ e^{2u} = \frac{H_{\Phi(\mathbb{S}^2)}\circ \Phi}{\left(\bar{\lambda}_1(g)^2 \Phi_0^2 + \bar{\lambda}_2(g)^2 \sum_{j\geq 1}\Phi_j^2\right)^{\frac{1}{2}} } $$
where $H_{\Phi(\mathbb{S}^2)}$ is the mean curvature of $\Phi(\mathbb{S}^2)$ in $\R^n$ and $\Phi_j$ for $j\geq 0$ are the coordinate functions of $\Phi$. Moreover, $\Phi_0$ is an eigenfunction with respect to $\lambda_1(g)$ and $\{ \Phi_j \}_{j\geq 1}$ are independent eigenfunctions with respect to $\lambda_2(g)$, and we have that
\begin{equation} \label{eqmassescriticalmetrics} \int_{\mathbb{S}^2} \Phi_0^2 dA_g =  \frac{t_1 A_g\left(\mathbb{S}^2\right)}{t_1 \bar{\lambda}_1(g) + t_2 \bar{\lambda}_2(g)} \text{ and } \int_{\mathbb{S}^2} \sum_{j \geq 1} \Phi_j^2 dA_g =  \frac{t_2 A_g\left(\mathbb{S}^2\right)}{t_1 \bar{\lambda}_1(g) + t_2 \bar{\lambda}_2(g)} \end{equation}
where $t_j = \left\vert \partial_i F(\bar{\lambda}_1(g),\bar{\lambda}_2(g)) \right\vert$ for $j=1,2$.
\end{theo}

The existence of such a map $\Phi$ was proved in \cite{Pet20} in a much more general setting (see also \cite{PT23}). However, in this general setting, we cannot obtain more than (possibly branched) minimal \textit{immersions} into one ellipsoid of \textit{possibly high dimension}. In fact, we noticed in \cite{Pet20} that the dimension of the target ellipsoid has to be lower bounded by $3$ because $\Phi$ is an immersion, and that since the coordinate functions are independent, the multiplicity of eigenvalues bounds the dimension of the target ellipsoid. By \cite{HON} we know that the multiplicity of first and second eigenvalues is bounded by $3$. Therefore, the dimension of the target ellipsoid is equal to $2$ if $\lambda_1(g) =\lambda_2(g)$ or if $\lambda_1(g) \neq \lambda_2(g)$ and $\lambda_2(g)$ has multiplicity $2$, and it is equal to $3$ if $\lambda_1(g) \neq \lambda_2(g)$ and $\lambda_2(g)$ has multiplicity $3$. 

The main novelty of Theorem \ref{theocriticalembedded} is that the map $\Phi$ has to be \textit{immersed} without branched points because the coordinates are first and second eigenfunctions. This result is proved in section \ref{sectionembeddedsphere}

Critical metrics for combinations of first and second eigenvalues are then good candidates to obtain non equatorial minimal embeddings into a 3-d ellipsoid. Minimizing the functional $g\mapsto F\left(\bar{\lambda}_1(g), \bar{\lambda}_2(g) \right) $ is the simplest way to construct them. We set
$$ \Lambda_F(\mathbb{S}^2) = \inf_{g} F\left(\bar{\lambda}_1(g), \bar{\lambda}_2(g) \right) \hskip.1cm,$$
where the infimum holds with respect to any metric $g$ on the sphere $\mathbb{S}^2$. 

\begin{theo} \label{theomin}
If 
\begin{equation} \label{mainassumption} \Lambda_F(\mathbb{S}^2) < \lim_{\eps\to 0}F(\eps,16\pi) \end{equation}
 then there is a smooth metric $g$ on $\mathbb{S}^2$, which is a minimizer for $\Lambda_F(\mathbb{S}^2)$. \end{theo}

This theorem was proved in \cite{Pet20} in a much more general setting (see also \cite{Pet22}). The sufficient condition \eqref{mainassumption} of Theorem \ref{theomin} is optimal since in the example $F(x_1,x_2) =\frac{1}{x_2}$, we know that there is not any minimizer for $\Lambda_F$ and we have the case of equality in \eqref{mainassumption}. 

Notice that \eqref{mainassumption} is already satisfied for many functions $F$. For instance it is automatic for functions such that $\lim_{\eps\to 0} F(\eps,16\pi) = +\infty$. 

Thanks to Theorem \ref{theoasymptotic}, we obtain much more applications $F$ satisfying the condition \eqref{mainassumption}.
One application is the existence of a maximizer of linear combinations of the first and second eigenvalues. Indeed, we set $F=f_t$ for $t>0$, where $f_t(x_1,x_2) = \frac{1}{x_1 +t x_2} $, it follows from Theorem \ref{theoasymptotic} that for some $\eps>0$ small enough 
$$f_t(0,16\pi) = \frac{1}{16\pi t} > f_t( \bar{\lambda}_1(g_{\eps}),\bar{\lambda}_2(g_{\eps}) ) \geq \Lambda_{f_t}(\mathbb{S}^2)$$
which is exactly the needed assumption \eqref{mainassumption} for Theorem \ref{theomin}. Of course, this Theorem \ref{theoasymptotic} applies for much more functions $F$ since the scale  $48\pi \eps^2$  converges to $0$ drastically faster than $\frac{4\pi}{\ln\frac{1}{\eps}}$ as $\eps\to 0$.

In the following proposition, we classify all the possible critical metrics corresponding to $2$-d ellipsoids, that is when the map $\Phi$ given by Theorem \ref{theocriticalembedded} is an isometry.  We will then obtain, a wild number of functions $F$ such that the target ellipsoid corresponding to minimizers cannot be 2-dimensional:

\begin{prop} \label{propE3}
We assume that $g$ is a critical point of $g \mapsto F\left(\bar{\lambda}_1(g), \bar{\lambda}_2(g) \right)$. We let $p := \bar{\lambda}_1(g)$ and $q := \bar{\lambda}_2(g)$. If there is one map $\Phi$ given by Theorem \ref{theocriticalembedded} such that the target ellipsoid is 2-dimensional, then

Either $p = q$, and $g$ is a round metric

or $p\neq q$ and $(\mathbb{S}^2,g)$ is isometric to $\mathcal{E}_{p,q} := \{ p x_0^2 + q x_1^2 +q x_2^2 = 1 \}$ endowed with the metric $g_{p,q} := \frac{H_{\mathcal{E}_{p,q}}}{\left(p^2 x_0^2 + q^2 x_1^2 + q^2 x_2^2 \right)^{\frac{1}{2}}} \xi$ of area $A_{p,q}$ such that $p$ is a first eigenvalue and $q$ is a second eigenvalue with respect to $g_{p,q}$ such that
\begin{equation} \label{eqcondderivative} \frac{\partial_2F(p,q)}{\partial_1F(p,q)} =  
\tau\left(\frac{q}{p}\right) \end{equation}
where $\tau$ is given in formula \eqref{eqdeftau}.

If in addition $g$ is a minimizer for $\Lambda_F\left(\mathbb{S}^2\right)$, then
\begin{equation} \label{eqcondenergy}  \Lambda_F(\mathbb{S}^2) = 
 \begin{cases} F(8\pi,8\pi) \text{ if } p=q \\
 F( p A_{p,q} , q A_{p,q} ) \text{ if } p\neq q \end{cases}  \end{equation}
\end{prop} 

In condition \eqref{eqcondderivative} we just explicitly compute the quotient $\frac{\int_{\mathbb{S}^2} x_0^2 dA_g}{A_g(\mathbb{S}^2) }$ for the 2-d ellipsoid $\mathcal{ E }_{p,q}$ and apply the last condition of Theorem \ref{theocriticalembedded}. 
By numerical computations, $\tau$ is a decreasing function and the condition \eqref{eqcondderivative} can only be realized when $\tau\left(\frac{p}{q}\right) \in ]t_1,2] $ for a number $t_1 \sim 0,2162$. 
More precisely, we have that
\begin{itemize}
\item $\tau(1) = 2$ corresponds to the round sphere $(p=q)$: in this case the round sphere is the unique possible 2-d ellipsoid corresponding to a critical metric. 
\item if $\tau\left(\frac{p}{q}\right) \in ]2,+\infty[$, then $q < p$, which means that $p$ cannot be a first or second eigenvalue on $\mathcal{E}_2$ since $q$ has multiplicity at least $2$.
\item if $\tau\left(\frac{p}{q}\right) \in ]0,t_1]$, then $p > q$ and $q A_{p,q} \geq 16\pi $ which means that the parameter $q$ is not a first or second eigenvalue on $\mathcal{E}_2$.
\end{itemize}
We want to prevent the conditions \eqref{eqcondderivative} or \eqref{eqcondenergy} from occuring. Of course, non-validity of one of these conditions depends on the choice of $F$. Let's give examples.

\subsection{Linear combination of first and second eigenvalues}

In the natural case of maximization of linear combinations of the two first eigenvalues, we set $F=f_t$, where $f_t(x_1,x_2) = \frac{1}{x_1 +t x_2} $ and as already said, minimizers always exist. 
Moreover, by numeric computations and proposition \eqref{propE3} there are at most two 2-dimensional critical ellipsoids for $f_t(\bar{\lambda}_1,\bar{\lambda}_2)$: either the sphere or the unique ellipsoid $\mathcal{E}_{p_t,q_t}$ with parameters $(p_t,q_t,q_t)$ such that $t = \tau\left(\frac{q_t}{p_t}\right) $. By numeric computations we have for any $t$
$$f_t(A_{p_t,q_t} p_t , A_{p_t,q_t}q_t) = \left(A_{p_t,q_t}(p_t+t q_t )\right)^{-1} \geq \left(8\pi (1+t)\right)^{-1} = f_t(8\pi,8\pi),$$ 
with equality if and only if $t=2$. We deduce that if the minimizer is a 2-d critical ellipsoid, it has to be a round metric. Then, knowing that $\Lambda_{f_t} > 16\pi t$, minimizers of $\Lambda_{f_t}$ give 3-d target ellipsoids if $t \geq 1$. In particular, we have a minimal metric $g_1$ for $t=1$, we can test it for the variational problem $\Lambda_{f_{1-\eps}}$. It gives:
$$ \Lambda_{f_{1-\eps}} - 16\pi (1-\eps) = \Lambda_{f_1} - 16\pi + \eps\left( \bar{\lambda}_2(g_1) + 16\pi\right) \geq 0 \text{ if } \eps \leq \frac{\Lambda_{f_1} - 16\pi}{\bar{\lambda}_2(g_1) + 16\pi}$$
so that we are sure to have a 3-d target ellipsoid also for $t \geq 1-\eps$. We deduce the existence of $0 \leq t_{\star} < 1-\eps$ such that the minimal metric for $\Lambda_{f_t}$ is a round sphere for $t = t_{\star}$ and all minimizers for $\Lambda_{f_t}$ are associated to 3-d ellipsoids for $t>t_{\star}$. 

\subsection{Linear combinations of the inverse of first and second eigenvalues}

We can also consider the example $h_s(x_1,x_2) = \frac{1}{x_1}+ s\frac{1}{x_2}$ for $s>0$. In the case $s \leq 2$, the round metric is the unique minimizer for $g \mapsto h_2(\bar{\lambda}_1(g),\bar{\lambda}_2(g)) = \frac{1}{\bar{\lambda}_1(g)} + s \frac{1}{\bar{\lambda}_2(g)}$. This is a consequence of the classical Hersch inequality \eqref{eqHersch}:
\begin{equation*}
\begin{split} h_s(\bar{\lambda}_1(g),\bar{\lambda}_2(g)) = \frac{1}{\bar{\lambda}_1(g)} + s \frac{1}{\bar{\lambda}_2(g)} = & \frac{s}{2}\left( \left(\frac{2}{s}-1\right)\frac{1}{\bar{\lambda}_1(g)} +  \frac{1}{\bar{\lambda}_1(g)} + \frac{2}{\bar{\lambda}_2(g)} \right) \\
 \geq & \frac{s}{2}\left( \left(\frac{2}{s}-1\right)\frac{1}{8\pi} +  \frac{1}{\bar{\lambda}_1(g)} + \frac{1}{\bar{\lambda}_2(g)} + \frac{1}{\bar{\lambda}_3(g)}\right) \\
 \geq & \frac{s}{2}\left( \left(\frac{2}{s}-1\right)\frac{1}{8\pi} + \frac{3}{8\pi} \right) =  \frac{1}{8\pi}\left(1+s\right) = h_s(8\pi,8\pi) \end{split}
\end{equation*}
holds for any metric $g$ and there is equality if and only if $g$ is the round metric. Let's consider the case $s>2$. From Theorem \ref{theomin}, there is a minimizer for $\Lambda_{h_s}$ for any $s$. In this case, by numeric computations and proposition \eqref{propE3}, there are at most two 2-d critical ellipsoids for $h_s(\bar{\lambda}_1,\bar{\lambda}_2)$: either the sphere or the unique ellipsoid $\mathcal{E}_{p_s,q_s}$ with parameters $(p_s,q_s,q_s)$ such that $s  = \frac{q_s^2}{p_s^2} \tau\left(\frac{q_s}{p_s}\right) $. By numeric computations we have $p_s < q_s$ if $s> 2$ and $p_s > q_s$ if $s<2$ and
\begin{equation} \label{eqcomparisonsphereellipsoid} h_s(A_{p_s,q_s} p_s , A_{p_s,q_s}q_s) = \frac{1}{p_s A_{p_s,q_s}}+ s \frac{1}{q_s A_{p_s,q_s}} \leq \frac{1}{8\pi} (1+t) = h_s(8\pi,8\pi) \end{equation}
with equality if and only if $s=2$. Then the minimizer cannot be a round sphere if $q_s$ is a second eigenvalue of $(\mathcal{E}_{p_s,q_s} , g_{p_s,q_s})$. It is clear that if $s$ is close to $2$, then $(p_s,q_s,q_s)$ is close to $(1,1,1)$, and $p_s$ and $q_s$ are first and second eigenvalue of $(\mathcal{E}_{p_s,q_s} , g_{p_s,q_s})$ by continuity of eigenvalues. However, by numeric computations, we observe that $x \mapsto x^2\tau(x)$ is an increasing function, and we obtain some number $s_1 \sim 7,0757$ such that, condition \eqref{eqcondderivative} is not realized for $s\in [s_1,+\infty[$. Indeed, for $s\geq s_1$, $q_s$ cannot be a second eigenvalue. Then there is $2 \leq s_{\star} \leq s_1$ such that for any $s\in [2,s_{\star}] $, $q$ is a second eigenvalue and for $s\in ]s_{\star},s_{\star}+\eps]$, $q$ is not a second eigenvalue anymore. We deduce that there is some $s \in [s_{\star},s_{\star}+\eps]$ having minimizers for $\Lambda_{g_s}$ correspond to minimal embeddings into $3$-d ellipsoids. Indeed, if not, by inequality \eqref{eqcomparisonsphereellipsoid} which is strict for $s_{\star} \neq 2$, the critical ellipsoid of parameters $(p_{s_\star},q_{s_{\star}},q_{s_{\star}})$ is a minimizer: $\Lambda_{h_{s_{\star}}} = h_{s_{\star}}(\bar{\lambda}_{1,\star},\bar{\lambda}_{2,\star})$ where $\bar{\lambda}_{1,\star} = p_{s_{\star}} A_{p_{s_{\star}},q_{s_{\star}}}$ and $\bar{\lambda}_{2,\star} = q_{s_{\star}} A_{p_{s_{\star}},q_{s_{\star}}})$, and for $s\in]s_{\star},s_{\star}+\eps[$, we have $\Lambda_{h_{s}} = h_s(8\pi,8\pi)$. We obtain a contradiction
\begin{equation*} 
\begin{split} 
\frac{1}{8\pi}(1+s) = h_s(8\pi,8\pi) = & \Lambda_{h_s} \leq  h_s( \bar{\lambda}_{1,\star},\bar{\lambda}_{2,\star}) \\
\leq & h_{s_{\star}}(\bar{\lambda}_{1,\star},\bar{\lambda}_{2,\star}) + \frac{s-s_\star}{\bar{\lambda}_{2,\star}} \\
\leq & \Lambda_{h_{s_{\star}}} - h_{s_{\star}}(8\pi, 8\pi) + \frac{1}{8\pi}(1+s_{\star}) + \frac{s-s_\star}{\bar{\lambda}_{2,\star}} \\
\leq &  \frac{1}{8\pi}(1+s) + \left(  \Lambda_{h_{s_{\star}}} - h_{s_{\star}}(8\pi, 8\pi) \right) + (s-s_{\star}) \left( \frac{1}{\bar{\lambda}_{2,\star}} - \frac{1}{8\pi}   \right),
\end{split}
\end{equation*}
choosing $s-s_{\star}$ small enough such that $ (s-s_{\star}) \left( \frac{1}{\bar{\lambda}_{2,\star}} - \frac{1}{8\pi}   \right) \leq  \frac{1}{2}\left( h_{s_{\star}}(8\pi, 8\pi) - \Lambda_{h_{s_{\star}}}  \right) $ which is possible since the right-hand term is positive by inequality \eqref{eqcomparisonsphereellipsoid} which is strict for $s_{\star} \neq 2$.

\subsection{Optimization of combinations of eigenvalues under symmetries}
In the current section we briefly explain that embedded immersions in Theorem \ref{theomain}, are obtained by an adaptation of the previous variational methods on combinations of eigenvalues to conformal factors on the sphere with \textit{symmetries}. We need to do that because we only have embeddedness under symmetry assumptions (see Theorem \ref{theoembedded}). To that purpose, we can follow the new point of view given in \cite{Pet22}: we build a sequence of \textit{symmetric} Palais-Smale sequences of conformal factors in the admissible space of  rotationally symmetric with respect to the third axe and symmetric with respect to $\{x_3=0\}$ conformal factors. This is left to the reader, noticing that more details are written in the analogous Steklov case (\cite{Pet23}). Bubble tree convergence of Palais-Smale sequences is of course still obtained (see \cite{Pet22}), and all the previous needed gap assumptions are still valid for $f_t$ and $h_s$ because the test functions given in Theorem \ref{theoasymptotic} are in the admissible space of conformal factors.

The remaining convergence as varifolds in Theorem \ref{theoembedded} of the obtained critical embeddings when $t\to +\infty$ or $s\to +\infty$ is again a consequence of the analysis given for the bubble tree convergence of Palais-Smale sequences in \cite{Pet22}. Again, more details are given in the analogous Steklov case \cite{Pet23}.

\section{Critical metrics for combinations of Laplace eigenvalues and minimal surfaces into ellipsoids} \label{sectioncriticalmetrics}

\subsection{Minimal immersions into ellipsoids} \label{minimmellispoid}
In this part, we explain the link between minimal immersions into ellipsoids and the Laplace eigenvalue of the induced metric times some function of the coordinates of the minimal immersion.
Let $\mathcal{E} \subset \mathbb{R}^n$ be an ellipsoid of parameters $\Lambda = diag\left(\lambda_1,\cdots,\lambda_n\right)$, with $\lambda_i >0$, defined by
$$ \mathcal{E} = \{ (x_1,\cdots,x_n) \in \mathcal{E} ; \lambda_1 x_1^2 + \cdots + \lambda_n x_n^2 = 1 \} \hskip.1cm,$$
endowed with the induced metric of the Euclidean metric $\xi$. The outward normal of the ellipsoid is denoted by
$$ \nu = \frac{\Lambda x}{\left\vert \Lambda x \right\vert}$$ 
where
$$\left\vert \Lambda x \right\vert = \left( \sum_{i=1}^n \lambda_i^2 x_i^2\right)^{\frac{1}{2}} \hskip.1cm.$$
We compute the laplacian of $x$ on $\mathcal{E}$:
$$ \Delta_{\mathcal{E}} x = \Delta_{\mathbb{R}^n} x - \nabla^2 x\left( \nu, \nu \right) + H \nu = H \nu \hskip.1cm, $$
where 
$$ H = div\left( \nu \right) = \frac{\sum_{i\neq j} \lambda_i \lambda_j^2 x_j^2}{\left\vert \Lambda x \right\vert^{3}} $$
is the mean curvature of the ellipsoid at $x \in \mathcal{E}$. 

\medskip

Now, let $\Phi : \Sigma \to \mathcal{E}$ be a conformal immersion of a Riemannian compact surface without boundary $(\Sigma,g)$ into $\mathcal{E}$, a $n-1$ dimentional ellipsoid of parameter $\Lambda = \left(\lambda_1,\cdots,\lambda_n\right)$. Then there is a smooth positive function $e^{2u}$, such that $g = e^{2u} h$, where the $h = \phi^{\star}(\xi) $ and we have
$$ \Delta_g f = e^{-2u} \Delta_h f \hskip.1cm, $$
If in addition, the isometric immersion $\Phi : (\Sigma,h) \to \mathcal{E}$ satisfies
$$ \Delta_h \Phi = \frac{H \circ \Phi}{\left\vert\Lambda \Phi \right\vert} \Lambda\Phi \hskip.1cm, $$
which is exactly the assumption of being minimal into $\mathcal{E}$, then setting 
$$g = e^{2u} h  \hspace{1cm} e^{2u} = \frac{H \circ \Phi}{\left\vert \Lambda \Phi \right\vert} = \frac{\left\vert \nabla \Phi \right\vert^2_{\mathcal{E}}}{\left\vert \Lambda \Phi \right\vert^2} \hskip.1cm,$$
the coordinates of $\Phi$ are eigenfunctions on $(\Sigma,g)$ with eigenvalues $\lambda_1,\cdots,\lambda_n$:
$$ \Delta_g \Phi = \Lambda \Phi \hskip.1cm.$$
Another well-known characterisation of an immersion $\Phi : (\Sigma,g) \to \mathbb{R}^n$ to be minimal in $\mathcal{E}$ is harmonicity in $\mathcal{E}$ and conformality. $\Phi$ is harmonic in $\mathcal{E}$ if it is a critical point of the energy
$$ E(\Phi) = \frac{1}{2} \int_{\Sigma}\left\vert \nabla \Phi \right\vert_g^2 dA_g $$
under the constraint $\Phi(\Sigma) \subset \mathcal{E}$. The Euler-Lagrange characterization is
$$ \Delta_g \Phi \in \left(T_{\Phi}\mathcal{E}\right)^{\perp}$$
Then $\Delta_g \Phi = f \nu$ for some function $f$ and simply computing $f$ with $0 = \frac{1}{2} \Delta_g \left\vert \Phi \right\vert_{\mathcal{E}}^2$, we obtain the equation
$$ \Delta_g \Phi = \frac{\left\vert \nabla \Phi \right\vert_\Lambda}{\left\vert \Lambda \Phi \right\vert^2} \Lambda \Phi = \left\langle \nabla \Phi, \nabla \nu \right\rangle \nu \hskip.1cm.$$
Conformality is characterized by the vanishing of 
$$0 = \left\vert\nabla \Phi \right\vert_g^2 \frac{g}{2} - d\Phi \otimes d\Phi := \sum_{i=1}^n \left(\left\vert\nabla \Phi_i \right\vert_g^2 \frac{g}{2} - d\Phi_i \otimes d\Phi_i \right) \hskip.1cm.$$

\subsection{Minimal embeddings of spheres into a 2-dimensional ellipsoid}

Let's consider the simplest case: we assume that $\Phi$ is an isometric embedding of a $2$-sphere (which exists by the uniformization theorem) that is $\Sigma = \mathcal{E}_p$ endowed with $\xi$ of parameters $p = (p_1,p_2,p_3)$. Then
$$ H = p_1p_2p_3 \frac{\frac{1}{p_1}+\frac{1}{p_2} + \frac{1}{p_3} -  x_1^2 - x_2^2 - x_3^2 }{\left( p_1^2 x_1^2 + p_2^2 x_2^2 + p_3^2 x_3^2 \right)^{\frac{3}{2}} } $$
We notice that setting the metric conformal to $\xi$
$$ g_p = e^{2u} \xi \hspace{1cm} e^{2u} = \frac{H}{\left\vert \Lambda x \right\vert}  = p_1 p_2 p_3 \frac{\frac{1}{p_1}+\frac{1}{p_2} + \frac{1}{p_3} -  x_1^2 - x_2^2 - x_3^2 }{\left(p_1^2 x_1^2 + p_2^2 x_2^2 + p_3^2 x_3^2 \right)^{2} } \hskip.1cm,$$
then on $(\mathcal{E}_p,g_p)$ we have 
$$ \Delta_{g_p} x_i = p_i x_i $$
for all $i=1,2,3$. This means that the coordinate functions are eigenfunctions on $(\mathcal{E}_p,g_p)$ with eigenvalues $p_1,p_2,p_3$. However, these are not necessarily eigenfunctions associated to the first, second and third non-zero eigenvalue of $(\mathcal{E},g)$. Indeed, for degenerating ellipsoids, with $p_2 = p_3 =1$, we have that 
$$ A_{g_{p}}(\mathcal{E}_{p}) \sim \frac{2\pi^2}{\sqrt{p_1}} \to +\infty \hbox{ as } p_1 \to 0 \hskip.1cm. $$
Knowing that $\lambda_k(g) A_g(\mathcal{E}_{p_1,p_2,p_3})$ has to be bounded by $8\pi k$ for any $k$, if $\lambda_k(g) = p_2 = p_3 = 1$, then $k\to + \infty$ as $p_1 \to 0$. It would be interesting to know the value of 
$$ k_{i}(p) = \inf \{ k\in \mathbb{N} ; \lambda_k(\mathcal{E}_p,g_{p}) = p_i \} $$
for $i=1,2,3$ and $p= (p_1,p_2,p_3)$ on the Riemannian manifold $(\mathcal{E}_p,g_p)$ defined above. For $p$ such that $p_1 = p_2 = p_3$, $k_i(p)= 1$ since the ellipsoid is a round sphere. The points $p$ such that $k_i$ jumps are points of bifurcation of eigenvalues. What is the cone $\{ p\in \mathbb{R}^3\setminus\{0\} ; \forall i \in\{1,2,3\}, k_i(p) \leq 3 \}$ ?

\subsection{Computations on rotationally symmetric 2-dimensional ellipsoids}
Let $\mathcal{E}_{(p,q)} = \{p x_0^2 +  q x_1^2 + q x_2^2 = 1 \}$ some ellipsoid with $p,q \geq 0$, endowed with the metric
$$ g_{p,q} = e^{2u_{p,q}} \xi \hspace{1cm} e^{2u_{p,q}} = pq^2 \frac{\frac{1}{p}+ \frac{2}{q} -  x_0^2 - x_1^2 - x_2^2 }{\left(p^2 x_0^2 + q^2 x_1^2 + q^2 x_2^2  \right)^{2} } \hskip.1cm.$$
By parametrization on the sphere $(x_0,x_1,x_2) = \varphi(z_0,z_1,z_2) = (\frac{1}{\sqrt{p}}z_0, \frac{1}{\sqrt{q}} z_1 , \frac{1}{\sqrt{q}} z_2)$, we obtain
\begin{equation*} \begin{split} \varphi^\star g_{p,q} = q \frac{2p+ q - \left(  q z_0^2 + p z_1^2 + p z_2^2 \right) }{ \left(p z_0^2 + q z_1^2 + q z_2^2 \right)^{2} } \varphi^\star \xi = \frac{2p+ q - \left( q z_0^2 + p z_1^2 + p z_2^2 \right) }{ \sqrt{p} \left(p z_0^2 + q z_1^2 + q z_2^2 \right)^{\frac{3}{2}} } g_{\mathbb{S}^2} \\ = \frac{p+q - (q-p)z_0^2}{\sqrt{p} \left( q - (q-p) z_0^2 \right)^{\frac{3}{2}}} g_{\mathbb{S}^2} \end{split} \end{equation*}
By stereographic projection with respect to $(1,0,0)$, knowing that $z_0 = \frac{r^2-1}{r^2+1}$
\begin{equation*} \begin{split} \pi_N^\star \varphi^\star g_{p,q} =  \frac{(p+q)(r^2+1)^2 -  (q-p)(r^2-1)^2}{\sqrt{p}(q(r^2+1)^2-(q-p)(r^2-1)^2)^{\frac{3}{2}}} (1+r^2) \frac{4}{(1+r^2)^2} \xi \\
= \frac{p r^4+2q r^2+p}{\sqrt{p}(p r^4+2(2q-p)r^2 +p)^{\frac{3}{2}}} \frac{8}{(1+r^2)} \xi \end{split} \end{equation*}
so that the area of this ellipsoid is
$$ A_{p,q} = 2\pi \int_{0}^{+\infty} \frac{p r^4+2q r^2+p}{\sqrt{p}(p r^4+2(2q-p)r^2 +p)^{\frac{3}{2}}} \frac{8}{(1+r^2)}  rdr   = \frac{1}{p} A_{\frac{q}{p}}$$
where $A_{s}:=A_{1,s}$. Notice that for $p=q$, $A_{p,p}=\frac{8\pi}{p}$. This corresponds to the area of the sphere times the sum of principal curvatures times the radius.

Now, we also have that since $x_1^2 + x_2^2 = \frac{1}{q} \left(z_1^2 +z_2^2\right) = \frac{1}{q}\frac{4r^2}{(1+r^2)^2}$
$$ \int_{\mathcal{E}_p} (x_1^2+x_2^2) dg_p = 2\pi \frac{1}{q}\int_{0}^{+\infty} \frac{4r^2}{(1+r^2)^2}\frac{p r^4+2q r^2+p}{\sqrt{p}(p r^4+2(2q-p)r^2 +p)^{\frac{3}{2}}} \frac{8}{(1+r^2)} rdr = \frac{1}{p q}B_{\frac{q}{q}}$$
where
$$ B_{s} = 2\pi  \int_{0}^{+\infty} \frac{4r^2}{(1+r^2)^2}\frac{ r^4+2s r^2+1}{( r^4+2(2s-1)r^2 +1)^{\frac{3}{2}}} \frac{8}{(1+r^2)} rdr$$
From Theorem \ref{theocriticalembedded}, we have that $\frac{1}{p q}B_{\frac{q}{q}} = \frac{ t A_{p,q} }{p+tq}$, we obtain that $g_{p,q}$ is a critical metric of $g\mapsto F\left(\bar{\lambda}_1(g),\bar{\lambda}_2(g) \right)$ if and only if
$ \frac{\partial_2F(p,q)}{\partial_1F(p,q)} =  \tau\left(\frac{p}{q}\right) $ where
\begin{equation} \label{eqdeftau} \tau(s) = \frac{B_{s}}{A_{s}-B_{s}} \frac{1}{s} \end{equation}
and all the properties used in Section \ref{sectionpresentation} are checked for this function $\tau$ by numerical computations.

\section{Road to embeddedness and proof of Theorem \ref{theoembedded}} \label{sectionembeddedsphere}
Let $\Phi : \mathbb{S}^2 \to \mathcal{ E } $ be some (possibly branched) minimal immersion of a sphere into an ellipsoid. We know that $\Phi$ is conformal and that the coodinates of $\Phi$ are eigenfunctions on the sphere endowed with the (smooth, possibly with conical singularities) metric 
\begin{equation} \label{defgphi} g_{\Phi} = \frac{\left\langle \Lambda \nabla \Phi , \nabla \Phi \right\rangle_{g_0}}{\left\vert \Lambda \Phi \right\vert^2} g_{0} \end{equation}
where $g_0$ is the standard round metric and $\Lambda = diag\left(\lambda_1,\cdots,\lambda_m\right)$ is the parameter of the ellipsoid:
$$ \mathcal{ E } =\left\{ (x_1,\cdots,x_m) \in \mathbb{R}^m ; \left\langle \Lambda x , x \right\rangle := \lambda_1 x_1^2 +\cdots + \lambda_m x_m^2 = 1 \right\} . $$
Notice that there is no reason for $ g_{\Phi} $ to be the pull-back of the Euclidean metric by $\Phi$ except if $\Lambda= \lambda I_m$. We assume that the coordinates of $\Phi$ are first or second eigenfunctions of the Laplacian with respect to $g_{\Phi}$. By rearrangement of coordinates, we can assume that all the coordinates are independent functions in $L^2(\mathbb{S}^2,g_{\Phi})$. We aim at proving that $\Phi$ does not have any branch point and that $\Phi$ is embedded.

\subsection{Nodal sets of eigenfunctions}  \label{subsecnodalset} In this short section, we give the main ingredients of the proof. They come from the properties of nodal sets and nodal domains. We recall that the nodal set of an eigenfunction is the set of vanishing points of the function and that the nodal domains are the connected components of the complementary of the nodal set. 

As noticed for instance in \cite{HON}, we know the topology of the nodal sets of first and second eigenfunctions on the sphere. By the celebrated Courant nodal theorem, first eigenfunctions have exactly two nodal domains, and second eigenfunctions have two or three nodal domains. Moreover, the nodal set of eigenvalues is locally diffeomorphic to the zero set of a polynomial of two variables. Therefore the zero set of first eigenfunctions has to be an embedded circle. The zero set of second eigenfunctions has to be 
\begin{itemize}
\item either an embedded smooth circle, 
\item or a disjoint union of two embedded smooth circles,
\item or an $\infty$-shaped curve (a union of two embedded circles having a unique intersection at a singularity point and smooth outside this singularity point).
\end{itemize}

In particular, one simple but useful remark is that any continuous map $\alpha : [0,1] \to E_2$, where $E_2$ the vector space of second eigenfunctions, such that $\alpha(0) = - \alpha(1)$ must satisfy that there is $t \in [0,1]$ such that the nodal set of $\alpha(t)$ is an embedded smooth circle. Indeed, if not, this means that every eigenfunction in $\alpha([0,1])$ has three nodal domains. On the sphere, two of them are simply connected while the third one is not simply connected. By connectedness reasons, for any $t \in [0,1]$, the two simply connected nodal sets of $\alpha(t)$ have the same constant sign. By assumption this is not the case for $\alpha(0)$ and $\alpha(1) = -\alpha(0)$ and we obtain a contradiction.

By \cite{HON}, we also deduce that the multiplicity of the first eigenvalue is bounded by $3$ and the multiplicity of the second one is also bounded by $3$. Therefore, $\Phi$ has at most $4$ coordinates. Since $\Phi$ is a (possibly branched) conformal immersion, it also has at least $3$ coordinates.

Moreover, for any eigenfunction $\psi$, we have by the Hopf lemma that for any $x\in \psi^{-1}(0)$ 
\begin{equation} \label{hopflemmanodal} x \text{ is a smooth point of } \psi^{-1}(0) \Leftrightarrow \nabla \psi(x) \neq 0. \end{equation}

\subsection{Case of 2 dimensional target ellipsoids} We assume that $\Phi$ has $3$ coordinates. Since $\Phi$ is a (possibly branched) conformal harmonic immersion then $\Phi : \mathbb{S}^2 \to \mathcal{E}_2$ is a branched cover. If one of the coordinates is a first eigenfunction then the nodal set is a smooth embedded circle. If all the coordinates are second eigenfunctions, then up to a rotation of the (rotationally symmetric) ellipsoid, one of the coordinates must be an eigenfunction such that the nodal set is a smooth embedded circle (see Section \eqref{subsecnodalset}). This means that the degree of the branched cover is one. Then $\Phi$ is an embedding.

\subsection{Case of 3 dimensional target ellipsoids} Notice that the current subsection can be seen as an independent section, except that we prove in Step 1 that any (possibly branched) minimal immersion by first and second eigenfunctions do not have branched points. We give, natural assumptions (see Step 9) in order to prove that they are embedded. These assumptions are implied by the symmetric assumptions of Theorem \ref{theoembedded}. This section gives a general picture and justifies why we add these symmetries in the next subsection. 

We assume that $\Phi$ has $4$ coordinates. We denote $\Phi = (\phi_0,\phi_1,\phi_2,\phi_3)$ the coordinates of $\Phi$ and we assume that $\phi_0$ is a first eigenfunction with respect to $g_{\Phi}$. Then $\phi_1,\phi_2,\phi_3$ are second eigenfunctions. We have by definition of $\Phi:= (\phi_0,\eta)$ that
\begin{equation}\label{mapintoellipsoid} \lambda_1 \phi_0^2 + \lambda_2 \left\vert \eta \right\vert^2 = 1
\end{equation}
The strategy is to prove that the restriction of $\eta := (\phi_1,\phi_2,\phi_3) : \mathbb{S}^2 \to \R^3$ is an injective map on the nodal set of $\phi_0$ and on the two nodal domains of $\phi_0$. For that, we focus on the structure of nodal set $N_v$ of second eigenfunctions $z \mapsto \left\langle \eta(z), v \right\rangle$ for $v \in \mathbb{S}^2$. We define
$$ A_1 = \{ v \in \mathbb{S}^2 ;  N_v \text{ is an embedded circle } \} $$
$$ A_2 = \{ v \in \mathbb{S}^2 ;  N_v \text{ is the disjoint union of two embedded circles } \} $$
$$ B = \{ v \in \mathbb{S}^2 ; N_v \text{ is an $\infty$-shaped curve } \} $$
From section \ref{subsecnodalset}, we know that $\mathbb{S}^2$ is the disjoint union of $A_1$, $A_2$ and $B$, that $A_1$ and $A_2$ are open sets while $B$ is a closed set and that $A_1 \neq \emptyset$. More generally, two antipodal points of $\mathbb{S}^2$ cannot be in the same connected component of $A_2 \cup B$.

We also define
$$ C = \{ z \in \mathbb{S}^2 ; \text{ there is } v \in B, z \text{ is the unique singular point of } N_v \}. $$
In fact, we can characterize $B$ and $C$ thanks to \eqref{hopflemmanodal}: for $z \in \mathbb{S}^2$, we have in any conformal chart that
\begin{equation} \label{eqhopflemmadet} z \in C \Leftrightarrow \det\left( \eta(z) , \eta_x(z), \eta_y(z) \right) = 0. \end{equation}
Indeed $z \in C$ iff there is $v \in B$ such that $z$  is the unique singular point of $N_v$ iff there is $v \in B$ such that $\left\langle v,\eta(z) \right\rangle = 0$ and $\left\langle v, \nabla\eta(z) \right\rangle = 0$ iff $\{\eta(z) , \eta_x(z), \eta_y(z)\}^{\perp} \neq \{0\}$. We also deduce that
\begin{equation} \label{eqcharBwithC} B = \mathbb{S}^2 \cap \bigcup_{z\in \mathbb{S}^2} \{\eta(z) , \eta_x(z), \eta_y(z)\}^{\perp}  = \mathbb{S}^2 \cap  \bigcup_{z\in C} \{\eta(z) , \eta_x(z), \eta_y(z)\}^{\perp}. \end{equation}

Finally, we deduce from Section \ref{subsecnodalset} that for any $z \in C$, then the rank of the familly $\{\eta(z) , \eta_x(z), \eta_y(z)\}$ is equal to $2$. Indeed, if not then $\{\eta(z) , \eta_x(z), \eta_y(z)\}^{\perp}$ contains at least one great circle of the sphere. It would give existence of two antipodal points which are in the same connected component of $B$. This is a contradiction.

\medskip

\textbf{Step 1}: $\nabla \eta$ never vanishes and then $\Phi$ does not have any branched point. This is the immediate consequence of the previous remark. $\nabla \eta(z) = 0$ for some $z\in \mathbb{S}^2$ would imply that the rank of the family $\{\eta(z) , \eta_x(z), \eta_y(z)\}$ is less than $1$.

\medskip

Thanks to Step 1, $\Phi : \mathbb{S}^2 \to \mathcal{E}$ is an immersion and we can define $n : \mathbb{S}^2 \to \mathbb{S}^2$ a unit normal to $\Phi$. This means that in a conformal chart, for any $z \in \mathbb{S}^2$, the family $(\Phi_x(z),\Phi_y(z), n(z) , \Lambda \Phi(z))$ is orthogonal, where $\Lambda = diag(\lambda_1,\lambda_2,\lambda_2,\lambda_2)$. We denote $n = (n_0,\tilde{n})$, where $n_0$ is the first coordinate of $n$. Notice that 
\begin{equation} \det\left( \eta(z) , \eta_x(z), \eta_y(z) \right) = 0 \Leftrightarrow n_0(z) = 0. \end{equation}
giving from \eqref{eqhopflemmadet} a new characterization of $C = n_0^{-1} \left( \{0\} \right)$.
Up to take $-n$ we choose $n$ such that $\det\left( \eta(z) , \eta_x(z), \eta_y(z) \right)$ and $n_0(z)$ have the same sign for any $z\in \mathbb{S}^2$.

\medskip

We set 
$$ C' = \{ z \in \mathbb{S}^2 , \eta(z) \subset Span\left( \eta_x(z) , \eta_y(z) \right) \} .$$
Of course $C' \subset C$. 

\medskip

\textbf{Step 2}: $C$ is the disjoint union of $\left(C\cap \phi_0^{-1}(\{0\})\right)$ and $C'$. 

\medskip

\textbf{Proof of Step 2}: Indeed, we would like to prove that if $z \in C$, then
$$ \phi_0(z) = 0 \Leftrightarrow Span\left( \eta_x(z) , \eta_y(z) \right) \text{ is a line}.$$
If $Span\left( \eta_x(z) , \eta_y(z) \right)$ is a line, then we have that $(\star,0,0,0) \in Span\left( \Phi_x(z) , \Phi_y(z) \right) $. We know that $(\lambda_1 \phi_0(z) , \lambda_2\eta(z)) \in Span\left( \Phi_x(z) , \Phi_y(z) \right)^{\perp}$. This gives that $\phi_0(z)=0$.

\noindent Now, if $Span\left( \eta_x(z) , \eta_y(z) \right)$ is a plane, then since $z\in C$, $\eta(z) \in Span\left( \eta_x(z) , \eta_y(z) \right)$. We can write $\eta(z) = \alpha \eta_x(z) + \beta \eta_x(z)$ for some constants $\alpha,\beta \in \mathbb{R}$. Since $(\Phi_x,\Phi_y, \Lambda\Phi)$ is an orthogonal family, we obtain
$$ 1 - \sigma_1 \phi_0^2(z) = \sigma_2 \left\vert \eta(z) \right\vert^2 = \sigma_2 \left\langle \eta(z) , \alpha \eta_x(z) + \beta \eta_x(z) \right\rangle = - \sigma_2 \phi_0(z) \left( \alpha \phi_{0,x}(z) + \beta \phi_{0,y}(z) \right),$$
so that $\phi_0(z)\neq 0$.

\noindent Notice that $\left(C\cap \phi_0^{-1}(\{0\})\right)$ and $C'$ must be disjoint sets because if not, $z\in C' \cap \phi_0^{-1}(\{0\})$ would satisfy that the rank of the family $\{\eta(z) , \eta_x(z), \eta_y(z)\}$ is less than $1$. This is not possible.

\medskip

\textbf{Step 3}: For any $z\in C$, $n_0$ is sign changing at the neighborhood of $z$. In particular, we have that $\left(C\cap \phi_0^{-1}(\{0\})\right) =  \phi_0^{-1}\left(\{0\}\right)$ or $\left(C\cap \phi_0^{-1}(\{0\})\right) =  \emptyset$. Moreover, the set $\eta^{-1}(\{0\})$ is finite and cannot be isolated in $C'$.

\medskip

\textbf{Proof of Step 3}: 

\medskip We notice that $n_0$ satisfies an elliptic equation and we aim at applying the maximum principle. We know that $\Phi$ is a minimal immersion. Then $n$ is harmonic so that 
$$\Delta n \in Span\{ n, \Lambda \Phi \}.$$
We write it as 
$$ \Delta n = \left\vert \nabla n \right\vert^2 n + \beta \Lambda \Phi $$
where $\beta$ is computed thanks to the property $\Delta\left( \left\langle n,\Lambda \Phi \right\rangle\right) = 0$ as
$$ \left\vert \Lambda\Phi \right\vert^2 \beta = 2 \left\langle \nabla n,\nabla \Lambda \Phi \right\rangle - \frac{\left\langle \Lambda n , \Lambda \Phi \right\rangle}{\left\vert \Lambda\Phi \right\vert^2} \left\langle \Lambda \nabla \Phi,\nabla \Phi \right\rangle . $$
We notice that 
$$ \left\langle \Lambda \Phi,\Lambda n \right\rangle = \lambda_1^2 \phi_0 n_0 + \lambda_2^2 \left\langle \tilde{n},\eta \right\rangle = - \lambda_1(\lambda_2-\lambda_1)n_0 \phi_0 $$ 
noticing that $\left\langle n,\Lambda \Phi \right\rangle = 0$ and that
$$ \left\langle \nabla n,\nabla \Lambda \Phi \right\rangle = \lambda_1 \nabla n_0 \nabla \phi_0 + \lambda_2 \left\langle \nabla \tilde{n},\nabla \eta \right\rangle = - \left(\lambda_2-\lambda_1\right)\nabla n_0 \nabla \phi_0 $$
noticing for the last inequality that since $\Phi$ is minimal, we have that the mean curvauture is equal to $0$, which implies that $ \left\langle \nabla n,\nabla \phi \right\rangle = 0 $. We obtain
$$ \Delta n_0 = \left( \left\vert \nabla n \right\vert^2 + \frac{\lambda_1^2(\lambda_2-\lambda_1)\phi_0^2}{\left\vert \Lambda \Phi \right\vert^4} \left\langle \Lambda \nabla \Phi,\nabla \Phi \right\rangle \right) n_0 - \frac{2\lambda_1(\lambda_2-\lambda_1)}{\left\vert \Lambda \Phi \right\vert^2} \phi_0 \nabla n_0 \nabla \phi_0 $$
Notice that 
$$ \frac{2\lambda_1(\lambda_2-\lambda_1)}{\left\vert \Lambda \Phi \right\vert^2} \phi_0  \nabla \phi_0 = \frac{\nabla\left( \lambda_1(\lambda_2-\lambda_1)\phi_0^2 \right)}{\lambda_2 -  \lambda_1(\lambda_2-\lambda_1)\phi_0^2 } = - \nabla\left( \ln \left\vert \Lambda \Phi \right\vert^2 \right) $$
where we used $\lambda_1 \phi_0^2 + \lambda_2 \left\vert \eta \right\vert^2 = 1$ to compute $\left\vert \Lambda \Phi \right\vert^2 = \lambda_2 -  \lambda_1(\lambda_2-\lambda_1)\phi_0^2 $. We obtain the equation
\begin{equation} \label{eqellipticn0} - div\left( \left\vert \Lambda \Phi \right\vert^{2} \nabla n_0 \right) = \left(\left\vert \Lambda \Phi \right\vert^2 \left\vert \nabla n \right\vert^2 + \frac{\lambda_1^2(\lambda_2-\lambda_1)\phi_0^2}{\left\vert \Lambda \Phi \right\vert^2} \left\langle \Lambda \nabla \Phi,\nabla \Phi \right\rangle \right) n_0 \end{equation}
and $z \in C$ is never a local maximum or minimum point of $n_0$.

To continue the proof of step 3, since $\phi_0^{-1}(\{0\})$ is a smooth embedded closed curve ($\phi_0$ is a first eigenfunction) and by Step 2, we deduce that $\mathbb{S}^2 \setminus C\cap \phi_0^{-1}(\{0\})$ is connected if and only if $C\cap \phi_0^{-1}(\{0\}) = \emptyset$ and that $\mathbb{S}^2 \setminus C\cap \phi_0^{-1}(\{0\})$ is disconnected if and only if $C\cap \phi_0^{-1}(\{0\}) = \phi^{-1}\left(0\right)$.

To complete the proof of step 2, we have that if $\eta^{-1}\left(\{0\}\right)$ is isolated in $C$, then it contains a closed curve. However, we must have that $\eta^{-1}\left(\{0\}\right)$ is finite. Indeed, if $\eta^{-1}\left(\{0\}\right)$ is infinite, by compactness of $\eta^{-1}\left(\{0\}\right)$, we can find an accumulation point $z_{\infty}\in \mathbb{S}^2$ such that $z_k \to z_{\infty}$ as $k\to +\infty$ and  $z_k \in \eta^{-1}\left(\{0\}\right)$. Up to a subsequence, there is $\tau \in T_{z_{\infty}}\mathbb{S}^2$ such that $\tau = \lim_{k\to +\infty} \frac{z_{\infty}-z_k}{\left\vert z_{\infty}-z_k \right\vert}$ and we get that $D\eta(z_{\infty})(\tau) = 0$. We obtain a contradiction because the set $\mathbb{S}^2 \cap \{ \eta_x(z_{\infty}), \eta_y(z_{\infty}),\eta(z_{\infty}) \} \subset B$ is a circle, giving existence of two antipodal points in the same connected component of $B$: a contradiction.

\medskip

\textbf{Step 4}: If $C = \emptyset$, then $\frac{\eta}{\left\vert \eta \right\vert} : \mathbb{S}^2 \to \mathbb{S}^2$ is a diffeomorphism and $\Phi$ is embedded.

\medskip

\textbf{Proof of Step 4}: Notice that 
$$ w := \left( \eta_x - \left\langle \frac{\eta}{\left\vert \eta \right\vert}, \eta_x \right\rangle  \frac{\eta}{\left\vert \eta \right\vert} \right)\wedge \left( \eta_y - \left\langle \frac{\eta}{\left\vert \eta \right\vert}, \eta_y \right\rangle  \frac{\eta}{\left\vert \eta \right\vert} \right) = \left\vert \eta \right\vert^2 \left( \frac{\eta}{\left\vert \eta \right\vert} \right)_x \wedge \left(  \frac{\eta}{\left\vert \eta \right\vert} \right)_y  $$
satisfies
$$ w = \eta_x \wedge \eta_y + \frac{\eta}{\left\vert \eta \right\vert^2} \wedge \left( \left\langle \eta,\eta_y \right\rangle \eta_x - \left\langle \eta,\eta_x \right\rangle \eta_y  \right) $$
so that
$$ \left\langle \eta,w \right\rangle = \left\langle \eta,\eta_x\wedge \eta_y \right\rangle $$
never vanishes if $C=\emptyset$. Therefore, $\frac{\eta}{\left\vert \eta \right\vert} : \mathbb{S}^2 \to \mathbb{S}^2$ is an immersion and it is a covering map since $\mathbb{S}^2$ is simply connected. Therefore it is a diffeomorphism. In particular, it is injective. This proves that $\Phi$ is injective and that $\Phi$ is an embedding. Notice that in this case, all the eigenfunctions have two nodal domains.

\medskip

\textbf{Step 5}: If $\phi_0^{-1}\left({\{0\}}\right) \subset C  $, then $\frac{\eta}{\left\vert \eta \right\vert} : \phi_0^{-1}\left({\{0\}}\right) \to \mathbb{S}^2 $ is an embedded curve, moreover, it is a convex curve in $\mathbb{S}^2$. 

\medskip

\textbf{Proof of Step 5}: 

\medskip

\textbf{Step 5.1}: $\frac{\eta}{\left\vert \eta \right\vert} : \phi_0^{-1}\left({\{0\}}\right) \to \mathbb{S}^2 $ is an immersed curve.

\medskip

\textbf{Proof of Step 5.1}: Let $z(t)$ with $t\in \mathbb{S}^1$ be an arc length parametrization of $\phi_0^{-1}(\{0\})$. Notice that $\left\vert \eta(z(t)) \right\vert = \lambda_2^{-1}$ for any $t$. We set $p(t) = \eta(z(t))$ and we compute in a conformal chart
$$ p'(t) = z_1'(t) \eta_x(z(t)) + z_2'(t) \eta_y(z(t)). $$
We know that $\phi_0(z(t)) = 0$ so that 
$$ z_1'(t) \phi_{0,x}(z(t)) + z_2'(t) \phi_{0,y}(z(t)) = 0. $$
and we deduce that up to take the parametrization $-z(t)$,
\begin{equation} \label{eqzprime} z_1'(t) = - \frac{\phi_{0,y}(z(t))}{\left\vert\nabla\phi_0\right\vert(z(t))} \text{ and } z_2'(t) =  \frac{\phi_{0,x}(z(t))}{\left\vert\nabla\phi_0\right\vert(z(t))} \end{equation}
where we know that $\left\vert\nabla\phi_0\right\vert(z(t)) \neq 0$ since $\phi_0(z(t)) =0$ and $\phi_0$ is a first eigenfunction.

Now, knowing that we have a conformal parametrization: 
$$ \left\vert \eta_y \right\vert^2 -  \left\vert \eta_x \right\vert^2 = \phi_{0,x}^2 - \phi_{0,y}^2 \text{ and } \left\langle \eta_x,\eta_y \right\rangle = -  \phi_{0,x} \phi_{0,y}$$
and that $\eta_x$ and $\eta_y$ are colinear vectors on $\phi_0^{-1}(\{0\}) \subset C$, we obtain that
$$ \eta_x = -\phi_{0,y} . \tau \text{ and } \eta_y  = \phi_{0,x} . \tau$$
for some vector $\tau$ such that $\left\vert \tau \right\vert = 1$. We obtain that
$$ p'(t) = \left\vert \nabla \phi_0 \right\vert(z(t)) . \tau(t) \neq 0 $$
so that $p$ is an immersion.

\medskip

\textbf{Step 5.2}: $\frac{\eta}{\left\vert \eta \right\vert} : \phi_0^{-1}\left({\{0\}}\right) \to \mathbb{S}^2 $ is a locally convex curve.

\medskip

\textbf{Proof of Step 5.2}: Again, let $z(t)$ with $t\in \mathbb{S}^1$ be an arc length parametrization of $\phi_0^{-1}(\{0\})$ and $p(t) = \eta(z(t))$. Without loss of generality, we assume that $\left\vert p(t) \right\vert = 1$. We set 
$$v(t) = p(t)\wedge \frac{p'(t)}{\left\vert p'(t) \right\vert}. $$
We aim at proving that the geodesic curvature $ \left\langle \left(\frac{p'(t)}{\left\vert p'(t)\right\vert}\right)', v(t) \right\rangle $ never vanishes. This is equivalent to prove that $ \left\langle p''(t), v(t) \right\rangle $ never vanishes. Since $\Phi$ is a minimal surface, we have the equations
\begin{equation} \label{eqsecondderivative}
\begin{cases} \Phi_{x,x} = u_x \Phi_x - u_y \Phi_y + e n + \frac{\left\vert \Phi_x \right\vert^2_\Lambda}{\left\vert \Lambda \Phi \right\vert^2} \Lambda\Phi \\
 \Phi_{y,y} = -u_x \Phi_x + u_y \Phi_y + g n + \frac{\left\vert \Phi_y \right\vert^2_\Lambda}{\left\vert \Lambda \Phi \right\vert^2} \Lambda\Phi \\
  \Phi_{x,y} = u_y \Phi_x + u_x \Phi_y + f n + \frac{\left\langle \Phi_x, \Phi_y \right\rangle_\Lambda}{\left\vert \Lambda \Phi \right\vert^2} \Lambda\Phi \\
\end{cases}
\end{equation}
where $e^{2u} = \left\vert \Phi_x \right\vert^2 = \left\vert \Phi_y \right\vert^2$ is the conformal factor and $e,f,g$ are the coefficients of the second fundamental form of the immersion into the ellipsoid. Since the immersion is conformal and minimal, we have $e+g = 0$.
In particular,
$$ \eta_{x,x} - e \tilde n \in Span(\eta,\eta_x,\eta_y) \text{ and } \eta_{y,y} + e \tilde n \in Span(\eta,\eta_x,\eta_y) \text{ and } \eta_{x,y} - f \tilde n \in Span(\eta,\eta_x,\eta_y) $$
Now, notice that
\begin{equation*} \begin{split} p''(t) - & \left( (z_1'(t))^2 \eta_{x,x}(z(t)) + 2z_1'(t)z_2'(t) \eta_{x,y}(z(t)) + (z_2'(t))^2 \eta_{y,y}(z(t)) \right) \\ & \in Span\left( \eta(z(t)), \eta_x(z(t)),\eta_y(z(t))\right) = \{v(t)\}^{\perp} = Span\left( p(t), p'(t)\right) \end{split} \end{equation*}
and that $\tilde{n}(t) = \pm v(t)$ and \eqref{eqzprime} to deduce that
$$  \left\langle p''(t), v(t) \right\rangle = \pm \frac{1}{\left\vert \nabla \phi_0 \right\vert^2} \left(e\left( \phi_{0,y}^2 - \phi_{0,x}^2\right) - 2 f \phi_{0,x}\phi_{0,y}  \right).$$
Let's prove that this quantity never vanishes. Up to a rotation of the conformal parametrization, we assume that for a given $t$, $z'(t) = (1,0)$. This means that $\phi_{0,x}(z(t)) = \left\vert \nabla \phi_0 \right\vert(z(t)) $ and that $\phi_{0,y}(z(t)) = 0 $ (the vertical axis is tangent to $\phi_0^{-1}\left(\{0\}\right)$ at $z(t)$ ) so that
$$  \left\langle p''(t), v(t) \right\rangle = \mp e(t).$$
Since we know by assumption that $C = n_0^{-1}\left(\{0\}\right) = \phi_0^{-1}\left(\{0\}\right)$ it is clear that $n_{0,y}(z(t)) = 0$ and that $n_{0,x}(z(t))$ is the normal derivative of $n_0$ along the curve $C$ at $t$. Since $n_0$ satisfies the elliptic equation \eqref{eqellipticn0}, we know that $n_{0,x}(z(t)) \neq 0$. Now, we have that
$$ n_x = -e^{2u} e \Phi_x  -e^{2u} f \Phi_y + \left\langle n_x, \Phi \right\rangle \Lambda \Phi$$
and we deduce $ n_{0,x}(z(t)) = -e^{2u(z(t))} e(z(t)) \phi_{0,x}(z(t)) $ so that
$$  \left\langle p''(t), v(t) \right\rangle = \pm \frac{n_{0,x}(z(t))}{\phi_{0,x}(z(t))} e^{2u(z(t))} =\pm \frac{\left\vert \nabla n_0 \right\vert(z(t))}{\left\vert \nabla \phi_0 \right\vert(z(t))} \frac{\left\vert \nabla \Phi \right\vert^2(z(t))}{2} \neq 0. $$

\medskip

\textbf{Step 5.3}: $\frac{\eta}{\left\vert \eta \right\vert} : \phi_0^{-1}\left({\{0\}}\right) \to \mathbb{S}^2 $ is a globally convex curve. In particular, it is a simple curve.

\medskip

\textbf{Proof of Step 5.3}: We recall that $\{-v(t), v(t) \} \in B$ is the unique pair such that the nodal set of the function $z \mapsto \left\langle \eta(z), \pm v(t) \right\rangle$ is an $\infty$-shaped curve. Moreover, the unique double point of this curve is $z(t)$. In particular, the curve $t\mapsto v(t)$ is simple. 

Now, we deduce from Step 5.2 that $v$ is an immersion and that $v$ is locally convex. Indeed, we have that $(p,\frac{p'}{\left\vert p' \right\vert},v)$ is a Frenet frame of $p$, that is
$$\left\vert p \right\vert^2 = \left\vert v \right\vert^2 = 1 \text{ and } \left\langle p,v \right\rangle =  \left\langle p',v \right\rangle = 0  $$
We deduce that $ \left\langle p,v' \right\rangle = 0 $ so that $v'$ and $p'$ are colinear. We also have that 
$$ \left\langle p',v' \right\rangle = - \left\langle v'',p \right\rangle = - \left\langle p,v'' \right\rangle . $$
We deduce from the first equation that $v'$ never vanishes so that $v$ is an immersion and then that $\left\langle v'', v\wedge v' \right\rangle $ never vanishes so that $v$ is locally convex.

By a result by Little \cite{little}, we know that a simple locally convex curve on $\mathbb{S}^2$ has to be a globally convex curve. Then $v$ is a globally convex curve. This means that $v$ satisfies (up to take $-p$) for any $s \neq t $, $\left\langle v(s),p(t) \right\rangle > 0$ (and of course, for any $t$, $\left\langle v(t),p(t) \right\rangle = 0$). Now we also deduce that $p$ is a globally convex curve and that in particular, it is a simple curve.

\medskip 

\textbf{Step 6}: If $C = \phi_0^{-1}\left({\{0\}}\right)$, then $\Phi$ is embedded.

\medskip

\textbf{Proof of Step 6}: The maps $\frac{\eta}{\left\vert \eta \right\vert} : \phi_0^{-1}\left(\mathbb{R}^*_+\right) \to \mathbb{S}^2$ and $\frac{\eta}{\left\vert \eta \right\vert} : \phi_0^{-1}\left(\mathbb{R}^*_-\right) \to \mathbb{S}^2$ are immersions (on $\mathbb{S}^2\setminus C$, by definition of $C$, see the proof of Step 4) such that $\phi_0^{-1}\left(\mathbb{R}^*_+\right)$ and $\phi_0^{-1}\left(\mathbb{R}^*_-\right)$ have the topology of a disk and by Step 5 are parametrizations of a closed simple curve at the boundary. Therefore, these maps have to be embeddings. In particular, they are injective. Now, since $\Phi = (\phi_0,\eta)$, we deduce that $\Phi$ is injective. Therefore, $\Phi$ is an embedding.

\medskip 

\textbf{Step 7}: If $\eta^{-1}(\{0\}) = \emptyset$, then $C$ is a disjoint union of closed curves.

\medskip

\textbf{Proof of Step 7}: 

\medskip

\textbf{Step 7.1} We set the map defined in a conformal chart by 
$$\psi(z) = \det\left( \eta(z) , \eta_x(z), \eta_y(z) \right) = \left\langle \eta(z),\eta_x(z)\wedge \eta_y(z) \right\rangle .$$
For any $z\in \psi^{-1}(\{0\})$, $\psi$ is a submersion at $z$ if and only if $A(z) \neq 0 $ or $\eta(z)\neq 0$ where $A$ is the second fundamental form of $\Phi$.

\medskip

\textbf{Proof of Step 7.1} Notice that we have \eqref{eqsecondderivative}. Let $z\in \psi^{-1}(\{0\})$. Then 
\begin{equation*} 
\begin{split} \psi_x(z) & = \left\langle \eta(z),\eta_{xx}(z)\wedge \eta_y(z) \right\rangle + \left\langle \eta(z),\eta_x(z)\wedge \eta_{xy}(z) \right\rangle \\
& =  e  \left\langle \eta(z), \tilde{n}(z)\wedge \eta_y(z) \right\rangle + f \left\langle \eta(z),\eta_x(z)\wedge \tilde{n}(z) \right\rangle
\end{split} \end{equation*}
and
\begin{equation*} 
\begin{split} \psi_y(z) & = \left\langle \eta(z),\eta_{xy}(z)\wedge \eta_y(z) \right\rangle + \left\langle \eta(z),\eta_x(z)\wedge \eta_{yy}(z) \right\rangle \\
& =  f  \left\langle \eta(z), \tilde{n}(z)\wedge \eta_y(z) \right\rangle - e \left\langle \eta(z),\eta_x(z)\wedge \tilde{n}(z) \right\rangle
\end{split} \end{equation*}
Therefore, if $ \psi(z) =  \psi_x(z) =  \psi_y(z)  = 0$, we obtain
\begin{equation} \label{eq2cond}e=f=0  \text{ or } \left\langle \eta(z), \tilde{n}(z)\wedge \eta_y(z) \right\rangle = \left\langle \eta(z),\eta_x(z)\wedge \tilde{n}(z) \right\rangle = 0 \end{equation}
which means that
$$ A(z) = 0 \text{ or } \eta(z) = 0 $$
Indeed, since $\psi(z)=0$, either $\phi_0(z) = 0$ and we must have that $\eta(z) \in \{ \tilde{n}(z),\eta_x(z),\eta_y(z) \}^{\perp}$, or $\phi_0(z) \neq 0$ and $\eta(z) \in Span  \{ \eta_x(z),\eta_y(z) \} $. Since $\tilde{n}(z) \neq 0$ and $\nabla \eta(z) \neq 0$, the second condition in \ref{eq2cond} gives that if $A(z)\neq 0$, then $\eta(z) = 0$.

\medskip

\textbf{Step 7.2} If $z\in C$, then $ A(z) \neq 0 $. 

\medskip

\textbf{Proof of Step 7.2} Let $z_0 \in C$. We look at the behaviour of $\eta$ at the neighbourhood of $z_0$. We have that
\begin{equation*} 
\begin{split}
 \eta(z_0 + (x,y)) = & \eta(z_0) + x \eta_x(z_0) + y \eta_y(z_0) \\
& + \frac{1}{2}\left( x^2 \eta_{xx}(z_0) + 2xy \eta_{xy}(z_0) + y^2 \eta_{yy}(z_0) \right) + o(\left\vert (x,y) \right\vert^2) 
\end{split} 
 \end{equation*}
Now, since $z_0\in C$, we have that $\{ \eta(z_0),\eta_x(z_0),\eta_y(z_0) \}^\perp = \R \tilde{n}(z_0) $ so that  using again \eqref{eqsecondderivative}, we have
$$ \left\langle\eta(z_0 + (x,y)), \tilde{n}(z_0)\right\rangle = \frac{1}{2}\left( e \left(x^2-y^2\right)  + 2 f xy \right) + o(\left\vert (x,y) \right\vert^2)$$
We know that the nodal set of the second eigenfunction $z \mapsto \left\langle \eta(z),\tilde{n}(z_0) \right\rangle$ is an $\infty$-shaped curve and in particular that at the singularity point $z_0$, the nodal set has looks like the nodal set of a harmonic polynomial of degree two. In particular, we must have $A(z_0)\neq 0$.

\medskip

\textbf{Step 8}: We write explicitely the Gauss-Codazzi equation

\medskip

\textbf{Proof of Step 8}: 

We write in complex coordinates the system of equations \eqref{eqsecondderivative}
\begin{equation} \label{eqsecondderivativecomplex}
\begin{split} \Phi_{zz} &  = \frac{1}{4} \left( \Phi_{x,x}-\Phi_{y,y} - 2i \Phi_{x,y} \right) \\
& = 2u_z \Phi_z + \frac{A}{2} n +  \left(\lambda_2-\lambda_1\right)\left(\phi_{0,z}\right)^2 \frac{\Lambda\Phi}{\left\vert \Lambda\Phi \right\vert^2} 
\end{split} \end{equation}
where $A = e-if$ contains all the informations of the second fundamental form of the immersion into the ellipsoid. We also have the equation
$$ \Phi_{z\bar{z}} = - \frac{1}{4}  \frac{\left\vert \nabla \Phi \right\vert^2_\Lambda}{\left\vert \Lambda\Phi \right\vert^2} \Lambda\Phi$$
We deduce that
$$ \left\langle \Phi_{zz\bar{z}} ,n  \right\rangle = \frac{1}{2} A_{\bar{z}} +(\lambda_2-\lambda_1) \left(\phi_{0,z}\right)^2 \frac{ \left\langle \Lambda \Phi_{\bar{z}},n\right\rangle}{\left\vert \Lambda\Phi \right\vert^2} $$
and that
$$ \left\langle \Phi_{z\bar{z}z} ,n  \right\rangle = - \frac{1}{4}  \frac{\left\vert \nabla \Phi \right\vert^2_\Lambda}{\left\vert \Lambda\Phi \right\vert^2} \left\langle\Lambda\Phi_z,n\right\rangle$$
and knowing that $\left\langle\Lambda\Phi_z,n\right\rangle = -(\lambda_2-\lambda_1) \phi_{0,z}n_0$ we deduce
$$ \frac{1}{2} A_{\bar{z}} = \left(\lambda_2-\lambda_1\right) \frac{n_0}{\left\vert \Lambda\Phi \right\vert^2} \left( \frac{1}{4}  \left\vert \nabla \Phi \right\vert^2_\Lambda \phi_{0,z} + \left(\lambda_2-\lambda_1\right) \left(\phi_{0,z}\right)^2 \phi_{0,\bar{z}}  \right) $$
and since $4 \phi_{0,z}\phi_{0,\bar{z}} = \left\vert \nabla \phi_0 \right\vert^2$ and $ \left\vert \nabla \Phi \right\vert^2_\Lambda = 2\lambda_2 e^{2u} - \left(\lambda_2-\lambda_1\right) \left\vert \nabla \phi_0 \right\vert^2 $, denoting by $\omega = \ln\left(\left\vert \Lambda \Phi \right\vert\right)$,
we obtain the following Gauss-Codazzi equation, 
\begin{equation} \label{eqgausscodazzi}
e^{-2u}e^{2\omega} A_{\bar{z}} = \lambda_2 \left(\lambda_2-\lambda_1\right) n_0 \phi_{0,z}.
\end{equation}
We also deduce from \eqref{eqsecondderivativecomplex} an equation on $\phi_0$
$$ \phi_{0,zz} = 2 u_z \phi_{0,z} + \frac{A}{2} n_0 + \left(\lambda_2-\lambda_1\right)\frac{ \lambda_1 \phi_0}{\left\vert \Lambda\Phi\right\vert^2} \left(\phi_{0,z}\right)^2 $$
but knowing that $ e^{2\omega} = \left\vert \Lambda\Phi\right\vert^2 = \lambda_2-\left(\lambda_2-\lambda_1\right)\lambda_1 \phi_0^2 $, we obtain
$$  \phi_{0,zz} = 2 u_z \phi_{0,z} + \frac{A}{2} n_0 - \omega_z \phi_{0,z} $$ so that
\begin{equation} \label{eqonphi0}
\left( e^{-2u} e^{\omega} \phi_{0,z}\right)_z = e^{-2u} e^{\omega} \frac{A}{2} n_0.
\end{equation}
Now, from the equations on $n$
$$ \begin{cases}
n_{x} = - e^{-2u} \left(e  \Phi_x + f \Phi_y\right) - \frac{ \left\langle \Lambda \Phi_x, n \right\rangle }{\left\vert\Lambda\Phi\right\vert^2} \Lambda \Phi  \\
n_{y} = -e^{-2u} \left( f \Phi_x - e \Phi_y\right) - \frac{ \left\langle \Lambda \Phi_y, n \right\rangle }{\left\vert\Lambda\Phi\right\vert^2} \Lambda \Phi
\end{cases} $$
we obtain that
\begin{equation} 
\begin{split}
 n_{0,z} &= -e^{-2u} A \phi_{0,\bar{z}} + \left(\lambda_2-\lambda_1\right) \phi_{0,z} n_0 e^{-2\omega} \lambda_1 \phi_0 \\
 &= -e^{-2u} A \phi_{0,\bar{z}} - \omega_z n_0
\end{split} 
  \end{equation}
we deduce that
\begin{equation}\label{eqonn0} e^{-\omega}\left(e^{\omega} n_{0}\right)_z = -e^{-2u} A \phi_{0,\bar{z}}. \end{equation}
We derive $e^{\omega}\times$\eqref{eqonn0} with respect to $\bar{z}$ and then we use \eqref{eqgausscodazzi} for the first right-hand term and \eqref{eqonphi0} for the second right-hand term. We obtain that
\begin{equation*}
\begin{split}
\left(e^{\omega} n_{0}\right)_{z \bar{z}} & = -  A_{\bar{z}} e^{-2u}e^{\omega} \phi_{0,\bar{z}} - A \left( e^{-2u}e^{\omega} \phi_{0,\bar{z}} \right)_{\bar{z}}  \\
 & = -e^{-\omega} \phi_{0,z}\phi_{0,\bar{z}}\lambda_2 (\lambda_2-\lambda_1) - e^{-2u}e^{\omega} n_0 \frac{\left\vert A \right\vert^2}{2} 
\end{split} 
 \end{equation*}
Therefore at any point, we obtain that 
$$ \Delta\left( e^{\omega}n_0\right) = \left( 2 e^{-2u}\left\vert A \right\vert^{2} + \lambda_2 \left(\lambda_2-\lambda_1\right) e^{-2\omega} \left\vert \nabla\phi_0 \right\vert^2 \right)e^{\omega}n_0$$

\medskip

\textbf{Step 9}: We assume that $n_{0,x} \phi_{0,y} - n_{0,y} \phi_{0,x} = 0$. Then $\eta^{-1}(\{0\}) = \emptyset$. Moreover, by step 7, we can denote $C' = \bigcup_{i =1}^N C_i $ the disjoint union of closed curves of $C' = C \setminus \phi_0^{-1}\left(\{0\}\right)$. Then for any $i\in \{1,\cdots N\}$, the curve $\frac{\eta}{\left\vert \eta \right\vert} : C_i \to \mathbb{S}^2$ is a convex curve. In particular, it is a simple curve.

\medskip

\textbf{Proof of Step 9}: 

\medskip

\textbf{Step 9.2} If $z\in C$, then $ \eta(z) \neq 0 $. 

\medskip

\textbf{Proof of Step 9.1} By Step 3, we know that $\eta^{-1}(\{0\})$ is finite but cannot be isolated in $C =  n_0^{-1} \left( \{ 0\} \right)$. By assumption, $n_{0,x} \phi_{0,y} - n_{0,y} \phi_{0,x} = 0$, so that $\phi_0$ is constant on connected components of $n_0^{-1} \left( \{ 0\} \right)$. By the relation $ \lambda_1 \phi_0^2 + \lambda_2 \left\vert \eta \right\vert^2 =1 $, we must have that $\eta^{-1}(\{0\}) = \emptyset$.

\medskip

\textbf{Step 9.2}: $\frac{\eta}{\left\vert \eta \right\vert} : C_i \to \mathbb{S}^2 $ is an immersed curve.

\medskip

\textbf{Proof of Step 9.2}: Let $z(t)$ with $t\in \mathbb{S}^1$ be an arc length parametrization of $C_i$. Notice that $\left\vert \eta(z(t)) \right\vert =$ is a constant function for any $t$ since the $\phi_0(z(t))$ is a constant function by assumption. We set $p(t) = \eta(z(t))$ and we compute in a conformal chart
$$ p'(t) = z_1'(t) \eta_x(z(t)) + z_2'(t) \eta_y(z(t)). $$
It is clear that $p'(t)\neq 0$ since $\eta_x(z(t)) \wedge \eta_y(z(t)) \neq 0$.

\medskip

\textbf{Step 9.3}: $\frac{\eta}{\left\vert \eta \right\vert} : C_i \to \mathbb{S}^2 $ is a locally convex curve.

\medskip

\textbf{Proof of Step 9.3}: The proof is exactly the same as the proof of Step 5.2, Notice here that by the Hopf lemma on the elliptic equation satisfied by $n_0$ \eqref{eqellipticn0}, $\nabla n_0(z(t)) \neq 0$. Notice also that by the equation \eqref{eqonn0}, at $z(t)$, knowing that $n_0(z(t))=0$ and that $A(z(t))\neq 0$ by Step 7.2, we deduce that $\nabla \phi_0(z(t)) \neq 0$

\medskip

\textbf{Step 9.4}: $\frac{\eta}{\left\vert \eta \right\vert} : C_i  \to \mathbb{S}^2 $ is a globally convex curve. In particular, it is a simple curve.

\medskip

\textbf{Proof of Step 9.4}: The proof is exactly the same as for the proof of Step 5.3.

\bigskip

We leave the following question to the reader:

\medskip

\textbf{Question}: Under the assumptions of Step 9, is $\frac{\eta}{\left\vert \eta \right\vert} : \phi_0^{-1}\left( \mathbb{R}^\star_+ \right) \to \mathbb{} \mathbb{S}^2$ embedded ? Or at least is $\Phi$ embedded ?

\subsection{Proof of Theorem \ref{theoembedded}}
We start with the following consequence of symmetry assumptions on $e^{2v}$ to symmetry properties of coordinates.
\begin{prop}
Let $\Phi = (\phi_0,\phi_1,\phi_2,\phi_3) : \mathbb{S}^2 \to \mathbb{R}^4$ be a (possibly branched) conformal minimal immersion into $\{ (x_0,x_1,x_2,x_3) \in \mathbb{R}^4 ; \lambda_1 x_0^2 + \lambda_2 (x_1^2+x_2^2+x_3^2) = 1 \}$. We set $\eta = (\phi_1,\phi_2,\phi_3)$ and
$$ e^{2v} = \frac{\lambda_1 \left\vert \nabla \phi_0 \right\vert^2 + \lambda_2 \left\vert \nabla \eta \right\vert^2 }{\lambda_1^2 \phi_0^2 + \lambda_2^2 \left\vert \eta \right\vert^2} $$
We assume that $\phi_0$ is a first eigenfunction and $\phi_1,\phi_2,\phi_3$ are independent second eigenfunctions of the Laplace operator with respect to $e^{2v} dA_{\mathbb{S}^2} $ and we assume that
$$\forall x \in \mathbb{S}^2, e^{2v}(x_1,x_2,x_3) = e^{2v}(-x_1,x_2,x_3) = e^{2v}(x_1,-x_2,x_3)= e^{2v}(x_1,x_2,-x_3). $$
Then, up to a rotation in the set of parametrizations $\mathbb{S}^2$ and up to a rotation of the coordinates of $\eta$, we have the following symmetries
$$ \forall x \in \mathbb{S}^2, \phi_0(x) = \phi_0(-x_1,x_2,x_3) = \phi_0(x_1,-x_2,x_3) = - \phi_0(x_1,x_2,-x_3) $$
$$ \forall x \in \mathbb{S}^2, \phi_1(x) = -\phi_1(-x_1,x_2,x_3) = \phi_1(x_1,-x_2,x_3) =  \phi_1(x_1,x_2,-x_3) $$
$$ \forall x \in \mathbb{S}^2, \phi_2(x) = \phi_2(-x_1,x_2,x_3) = -\phi_2(x_1,-x_2,x_3) =  \phi_2(x_1,x_2,-x_3) $$
$$ \forall x \in \mathbb{S}^2, \phi_3(x) = \phi_3(-x_1,x_2,x_3) = \phi_3(x_1,-x_2,x_3) = \phi_3(x_1,x_2,-x_3) $$
Moreover, for $i=0,1,2$, $\phi_i$ has two nodal domains and the nodal set is $\{ x_{i} = 0 \}$ (where we let $x_0 = x_3$) and $\phi_3$ has three nodal domains and does not vanish on $\{x_3 = 0\} \cup \{(0,0,\pm 1)\}$.
\end{prop}

\begin{proof}
For $j=1,2,3$, we denote $s_j : \mathbb{S}^2 \to \mathbb{S}^2$ the symmetry with respect to the plane $\{ x_j = 0\} $. An eigenfunction $\phi$, satisfying the equation
$$ \Delta_{\mathbb{S}^2} \phi = \lambda e^{2v} \phi $$
satisfies that for $j = 1,2,3$, $\phi \circ s_j$ is also a solution. Therefore, up to replace orthogonal (with respect to $L^2\left(e^{2v} dA_{\mathbb{S}^2} \right)$) families of eigenfunctions by non zero eigenfunctions among
\begin{equation*}
\begin{split} \phi + & \eps_1 (\phi\circ s_1) + \eps_2 (\phi\circ s_2) + \eps_3 (\phi \circ s_3) \\ + & \eps_1\eps_2 (\phi \circ s_1 \circ s_2) + \eps_2 \eps_3  (\phi \circ s_2 \circ s_3) + \eps_3 \eps_1 (\phi \circ s_3 \circ s_1) + \eps_1\eps_2\eps_3 (\phi \circ s_1 \circ s_2 \circ s_3)
\end{split} \end{equation*}
where $\eps_1, \eps_2, \eps_3 \in \{-1,1\}$, we obtain that 
$$\forall j \in \{1,2,3\}, \phi_0 = \eps_0^j (\phi_0 \circ s_j). $$
where $\eps_0^j \in \{-1,1\}$ since $\phi_0$ is a simple first eigenfunction and that there are orthogonal (with respect to $L^2\left(e^{2v} dA_{\mathbb{S}^2} \right)$) second eigenfunctions $\eta_1,\eta_2,\eta_3$ such that
$$ \forall i,j \in \{1,2,3\}, \eta_i = \eps_{i}^j (\eta_i \circ s_j). $$
where $\eps_i^j \in \{-1,1\}$. 

\medskip

\textbf{Step 1:} Let $i = 0,1,2,3 $, the set $\{ j \in \{1,2,3\} ; \eps_i^j = -1 \}$ has at most one integer.

\medskip

\textbf{Proof of step 1:} If one of these sets has at least two integers, then one of the functions $\phi_0$, $\eta_1$, $\eta_2$, $\eta_3$ would be antisymmetric with respect to two symmetries among $s_1,s_2,s_3$. This means that it would have at least 4 nodal domains: it is impossible.

\medskip

\textbf{Step 2:} Let $j= 1,2,3$, the set $\{ i \in \{0,1,2,3 \} ; \eps_i^j = -1 \}$ has at most one integer. 

\medskip

\textbf{Proof of step 2:} If for some $j$ two eigenfunctions among $\phi_0, \eta_1,\eta_2,\eta_3$ are antisymmetric with respect to $s_j$. In particular, they vanish on $\{x_j = 0\}$. Then since they are orthogonal, one of them has to vanish on $\mathbb{S}^2 \setminus \{ x_j = 0 \}$. This means that this eigenfunction has at least two nodal domains that belong to $\{ x_j > 0 \}$ and by symmetry at least two nodal domains in $\{ x_j < 0 \}$. We obtain at least 4 nodal domains: it is impossible.

\medskip 

\textbf{Step 3:} Let $i \in \{0,1,2,3\}$. If the set $A_j = \{ j \in \{1,2,3\} ; \eps_i^j = -1 \}$ is empty, then $i\neq 0$ and $\eta_i$ has exactly 3 nodal domains. Moreover, $ \{ j \in \{1,2,3\} ; A_j = \emptyset \}$ contains at most one integer.

\medskip

\textbf{Proof of step 3:} We prove that if a non constant first or second eigenfunction $\phi$ satisfies $\phi = \phi \circ s_j$ for any $j=1,2,3$, then $\phi$ has 3 nodal domains. First of all, because of the symmetries, $\phi$ has to vanish on $\mathbb{S}^2 \setminus \bigcup_{j=1,2,3} \{ x_j = 0\}$: let $x \in \mathbb{S}^2$ such that $x_j >0$ for any $j=1,2,3$ and $\phi(x) = 0$. Because of the symmetries again, $x$ cannot be a singular point of the nodal set of $\phi$. Let $C$ be the connected component of $ \{ \phi(x) = 0 \} \setminus \bigcup_{j=1,2,3} \{ x_j = 0\} $ containing $x$. If $C$ is a closed curve, then by symmetries, $\phi$ would have at least 9 nodal domains: it is impossible. Then, $C$ has two ends $p_1 , p_2 \in \{ x \in \mathbb{S}^2 ; \forall k , x_k \geq 0 : \exists j, x_j = 0 \} $. If one of the ends $p_1$ belongs to $\{ (1,0,0),(0,1,0),(0,0,1) \}$, then by symmetries, $p_1$ and $-p_1$ are two singular vanishing points: the eigenfunction has at least 4 nodal domains: it is impossible. Then, let $j \in \{1,2,3\}$ be such that $p_1,p_2 \notin \{ x_j = 0\}$. Then, by symmetries, there is a nodal domain $D$ such that $D \subset \{x_j >0 \}$ and by symmetries again $-D \subset \{ x_j < 0 \}$. $\phi$ has at least $3$ nodal domains.

By contradiction, we assume that $ \{ j \in \{1,2,3\} ; A_j = \emptyset \}$ contains at least two integers. Up to a permutation of $\eta_1,\eta_2,\eta_3$, we deduce that $\eta_1 $ and $\eta_2$ have exactly three nodal domains. Now we consider the map $ N : e^{i\theta} \in \mathbb{S}^1 \mapsto $ \textit{ the number of positive nodal domains of $\phi_{\theta}:= \cos\theta \eta_1 + \sin\theta \eta_2$}. Since for any $j=1,2,3$, $\phi_\theta = \phi_\theta \circ s_j$, we deduce that $\forall \theta \in \mathbb{S}^1$ $\phi_\theta$ has 3 nodal domains. Then $N$ has to be a constant function. It contradicts the fact that for any $\theta$, $\{ N(e^{i\theta}), N(-e^{i\theta}) \} = \{1,2\}$.

\medskip

\textbf{Conclusion of Step 1,2,3:} Up to a rotation on the set of parametrisations $\mathbb{S}^2$, we obtain that 
$$ \forall x \in \mathbb{S}^2, \phi_0(x) = \phi_0(-x_1,x_2,x_3) = \phi_0(x_1,-x_2,x_3) = - \phi_0(x_1,x_2,-x_3) $$
$$ \forall x \in \mathbb{S}^2, \eta_1(x) = -\eta_1(-x_1,x_2,x_3) = \eta_1(x_1,-x_2,x_3) =  \eta_1(x_1,x_2,-x_3) $$
$$ \forall x \in \mathbb{S}^2, \eta_2(x) = \eta_2(-x_1,x_2,x_3) = -\eta_2(x_1,-x_2,x_3) =  \eta_2(x_1,x_2,-x_3) $$
$$ \forall x \in \mathbb{S}^2, \eta_3(x) = \eta_3(-x_1,x_2,x_3) = \eta_3(x_1,-x_2,x_3) = \eta_3(x_1,x_2,-x_3) $$
and that $\phi_0$, $\eta_1$, $\eta_2$ have two nodal domains and that $\eta_3$ has $3$ nodal domains. Moreover, as noticed in step $3$, $\eta_3$ cannot vanish in $\{ (\pm 1,0,0); (0,\pm 1,0);(0,0,\pm 1) \} $.

\medskip

\textbf{Conclusion:} It remains to prove that up to a rotation of the coordinates $(\phi_1,\phi_2,\phi_3)$ and up to multiply $\eta_i$ by constants, then $\phi_i = \eta_i$ for any $i\in \{1,2,3\}$.

\medskip

First of all, up to multiply $\eta_3$ by a constant, and up to a rotation, we can assume that $\phi_3 = \eta_3$. Then, we set
$$ \phi_1 = A + \alpha \phi_3 \text{ and } \phi_2 = B + \beta \phi_3 $$
for some functions $A,B \in \left\langle \eta_1,\eta_2 \right\rangle$.  Up to a rotation and up to multiply $\eta_1$ by a constant, we can assume that 
$$ A = \eta_1  \text{ and } B = a\eta_1 + b \eta_2  $$
where we must have $b\neq 0$. Up to multiply $\eta_2$ by a constant, we assume that $b=1$. We have that
\begin{equation*} 
\begin{split}
1 &= \sigma_1 \phi_0^2 + \sigma_2\left(\phi_1^2 + \phi_2^2+ \phi_3^2 \right) \\
 & = \sigma_1 \phi_0^2 +  \sigma_2(1+\alpha^2 + \beta^2) \phi_3^2 + \sigma_2\left(\eta_1^2 + \eta_2^2 \right) + 2\sigma_2 a \eta_1\eta_2 + 2 \sigma_2( \alpha A + \beta B  ) \phi_3 
\end{split}
 \end{equation*}
We deduce that $ a \eta_1\eta_2 + ( \alpha A + \beta B  ) \phi_3$ is an even function with respect to $s_1$ and $s_2$. Then
$$a\eta_1\eta_2 + (\alpha + \beta a) \eta_1\phi_3 + \beta\eta_2\phi_3 
=-a \eta_1 \eta_2 - \left( \alpha + \beta a \right) \eta_1 \phi_3 + \beta \eta_2 \phi_3  $$
and
$$a\eta_1\eta_2 + (\alpha + \beta a) \eta_1\phi_3 + \beta\eta_2\phi_3 
=-a \eta_1 \eta_2 + \left( \alpha + \beta a \right) \eta_1 \phi_3 - \beta \eta_2 \phi_3  $$
By a difference between those equations, we obtain $ \left( \alpha + \beta a \right) \eta_1 \phi_3 = \beta \eta_2 \phi_3$ so that $\beta = 0$ and $\alpha + \beta a = 0$. Then, we have $a = 0$, so that $a = \alpha = \beta= 0$.
\end{proof}

\begin{prop}
Let $\Phi = (\phi_0,\phi_1,\phi_2,\phi_3) : \mathbb{S}^2 \to \mathbb{R}^4$ and let $ e^{2v} : \mathbb{S}^2 \to \mathbb{R}_+$.
We assume that $\phi_0$ is a first eigenfunction and $\phi_1,\phi_2,\phi_3$ are independent second eigenfunctions of the Laplace operator with respect to $e^{2v} dA_{\mathbb{S}^2} $ and we assume that
$$\forall x \in \mathbb{S}^2, e^{2v}(x_1,x_2,x_3) = e^{2v}(x_1,x_2,-x_3) = m(x_3) $$
$$ \forall x \in \mathbb{S}^2, \sigma_1 \phi_0^2 + \sigma_2 \left(\phi_1^2 + \phi_2^2 + \phi_3^2 \right) = 1 + n\left(\frac{(x_1,x_2)}{\left\vert (x_1,x_2) \right\vert}\right) $$
for a even function $m : [-1,1]\to \R$ and a function $n : \mathbb{S}^1 \to \R$ such that$\int_{\mathbb{S}^1} n = 0$. Then, $n = 0$ and up to a rotation of the coordinates of $\eta = (\phi_1,\phi_2,\phi_3)$, we have the following symmetries
$$ \forall x \in \mathbb{S}^2, \phi_0(x) = f(x_3)  $$
$$ \forall x \in \mathbb{S}^2, \phi_1(x) = g(x_3)x_1 $$
$$ \forall x \in \mathbb{S}^2, \phi_2(x) = g(x_3)x_2 $$
$$ \forall x \in \mathbb{S}^2, \phi_3(x) = h(x_3) $$
for an odd function $f:[-1,1]\to \R$ and even functions $g,h : [-1,1]\to \R$. Moreover, for $i=0,1,2$, $\phi_i$ has two nodal domains and the nodal set is $\{ x_{i} = 0 \}$ (where we let $x_0 = x_3$) and $\phi_3$ has three nodal domains and does not vanish on $\{x_3 = 0\} \cup \{(0,0,\pm 1)\}$ and
$\Phi$ is an embedded conformal non planar minimal embedding into $\{x\in \R^4 ; \sigma_1 x_0^2 + \sigma_2(x_1^2+x_2^2+x_3^2) =1  \}$
\end{prop}

\begin{proof}
By the previous proposition and knowing that the mean value of an eigenfunction $\phi$ on circles $\{ x_3 = c \}$ for $c \in \{-1,1\}$ is also an eigenfunction, $\phi_0$ has to be antisymmetric with respect to the coordinate $x_3$ and radial: there is an odd function $f$ such that for all $x\in \mathbb{S}^2$, $ \phi_0(x) = f(x_3) $. Similarly, by the previous proposition, up to a rotation of $(\phi_1,\phi_2,\phi_3)$, there is an even function $h$ such that for all $x\in \mathbb{S}^2$, $ \phi_3(x) = h(x_3) $. 

By the previous proposition, there are orthonormal second eigenfunctions (with respect to $e^{2v}dA_{\mathbb{S}2}$), $ \eta_1 $ and $ \eta_2 $ such that $\eta_1(x_1,x_2,x_3) = -\eta_1(-x_1,x_2,x_3)$ and $\eta_2(x_1,x_2,x_3) = -\eta_2(x_1,-x_2,x_3)$ and any orthonormal eigenfunction orthogonal to $\phi_1$ and $\phi_3$ has to be written
$$ \eta_1 \circ R_{\alpha} = \cos\left( \varphi(\alpha) \right) \eta_1 + \sin\left( \varphi(\alpha)\right) \eta_2 $$
and
$$ \eta_2 \circ R_{\alpha} = - \sin\left( \varphi(\alpha) \right) \eta_1 + \cos\left( \varphi(\alpha)\right) \eta_2 $$
for some function $\varphi$ since $\left(\eta_i \circ R_\alpha\right)_{i=1,2}$ is orthonormal with respect to $e^{2v}dA_{\mathbb{S}^2}$. We deduce that
$$ \eta_1^2 + \eta_2^2 = \left(\eta_1^2 + \eta_2^2\right)\circ R_\alpha $$
for any $\alpha$ so that we can set a function $g$ such that $g(x_3) = \left(\eta_1^2 + \eta_2^2\right)^{\frac{1}{2}}(x)$ for any $x\in \mathbb{S}^2$.

We denote for a function $\phi : \mathbb{S}^2 \to \R$, $\tilde \phi = \phi \circ \pi^{-1}$ where $\pi :\mathbb{S}^2\setminus\{0,0,1\} \to \mathbb{R}^2$ is the stereographic projection with respect to the north pole. We have functions $\tilde{f}$ and $\tilde{h}$ such that $\tilde{\phi}_0(z) = \tilde{f}(r)$ and $\tilde{\phi}_3(x) = \tilde{h}(r)$, where $z = (r \cos\theta, r\sin\theta) $. We also have a function $\tilde{g}$ such that 
$$ \tilde{\eta}_1(z) = \cos\left( \beta(r,\theta) \right) \tilde{g}(r) \text{ and } \tilde{\eta}_2(z) = \sin\left( \beta(r,\theta) \right) \tilde{g}(r) $$
Now, since $\widetilde{\eta_1\circ R_\alpha} ( r e^{i\alpha}) = 0 $, for any $r$  we have that
$$ \forall r, \beta(r,\alpha) = \varphi(\alpha). $$
Finally, using the second eigenvalue equation on $\eta_i$ in the polar coordinates the chart $\pi$:
$$- \left( \frac{1}{r}\partial_r\left(r \left(\partial_r \tilde{\eta}_i\right)\right) + \frac{1}{r^2} \partial_{\theta}^2 \tilde{\eta}_i \right) = \sigma_2 e^{2\tilde{v}} \tilde{\eta_i}$$
We must have that $\forall \theta, \varphi'(\theta) $ is a constant, and then that $\varphi(\theta) = \theta$. so that 
$$ \tilde{\eta}_1(z) = \cos \theta  \tilde{g}(r) \text{ and } \tilde{\eta}_2(z) = \sin\theta  \tilde{g}(r) $$

Then, we prove that up to multiply $\eta_1$ and $\eta_2$ by the same constant, $\phi_1 = \eta_1$ and $\phi_2 = \eta_2$. 
Up to rotations and up to multiply $\eta_1$ and $\eta_2$ be the same constant, we have $a,b \in \R$ and $\alpha,\beta\in \R$, with $b \neq 0$ such that
$$ \phi_1 = \eta_1 + \alpha \phi_3  \text{ and } \phi_2 = a \eta_1 + b \eta_2 + \beta \phi_3. $$
Then
\begin{equation*}
\begin{split} 
1 + n = & \sigma_1 \phi_1^2 + \sigma_2\left( (1+\alpha^2 + \beta^2) \phi_3^2\right) + \sigma_2 \left( \eta_1^2 (1+a^2) + \eta_2^2 b^2 + 2ab \eta_1 \eta_2 \right) \\
& + 2 \sigma_2 \phi_3 \left((\alpha + \beta a) \eta_1 + \beta b \eta_2\right)
\end{split} 
\end{equation*}
We apply the stereographic projection $\pi$, we divide by $\tilde{g}$ and make a derivation with respect to $r$ and $\theta$:
$$\tilde{g}'(r) \left( \sin(2\theta) (b^2 - a^2 - 1) + \cos(2\theta)b^2 \right) + \tilde{h}'(r) \left( -(\alpha+\beta a)\sin\theta + \beta b \cos \theta \right) $$
and since $\tilde{h}'$ and $\tilde{g}'$ are independent functions, 
$$ b^2 - a^2 - 1 = 0 \text{ and } b^2 = 0 \text{ and } \alpha+\beta a = 0 \text{ and } \beta b = 0 $$
so that $b= 0, a = \pm 1$ and $ \alpha + \beta a = 0 $. Then, we obtain that $n$ is a constant function, and then $n = 0$. By the previous proposition, $\alpha = \beta = 0$.

We deduce that $\Phi : \mathbb{S}^2 \to \mathcal{E}$ is a harmonic map. Since in the chart $\pi$, 
$$\left(\left\vert \tilde{\Phi}_x \right\vert^2 - \left\vert \tilde{\Phi}_y \right\vert^2- 2i \left\langle \tilde{\Phi}_x,\tilde{\Phi}_y \right\rangle\right)  dz^2$$ 
is a quadradic holomorphic form on $\mathbb{S}^2$, it has to be equal to $0$ and $\Phi$ is conformal.
\end{proof}

\begin{proof}[Proof of Theorem \ref{theocriticalembedded}] 
Notice first that $\Phi$ is an immersion (Step 1 in the previous subsection). It suffices to prove that $\Phi$ is injective.

We set $\tilde{\Phi} = \Phi \circ \pi$ in the conformal chart $\pi : \mathbb{S}^2 \setminus\{(0,0,1)\} \to \R^2$, we have that the orthogonal family 
$$(\sigma. \tilde{\Phi} , \tilde{\Phi}_r , \tilde{\Phi}_\theta , \tilde{n}) = 
\begin{pmatrix} \sigma_1 f & f & 0 & \tilde{n}_0 \\
\sigma_2 g \cos\theta & g' \cos\theta & -g \sin\theta & \tilde{n}_1 \\
\sigma_2 g \sin\theta & g'\sin\theta & g\cos\theta & \tilde{n}_2 \\
\sigma_2 h & h' & 0 & \tilde{n}_3
 \end{pmatrix}$$
satisfies that $\tilde{\phi}_0$ and $\tilde{n}_0$ are radial functions, that $\phi_0^{-1}(\{0\}) \subset  n_0^{-1}\left(\{0\}\right)$ and that $\tilde{n}_0$ is antisymmetric with respect to the inversion $z \mapsto \frac{z}{\left\vert z \right\vert^{2}}$. In fact, up to the multiplication of a positive radial function, we have that
$$ \tilde{n}_0 \sim g(g' h -h' g) \sim g \frac{g'h-h'g}{\sqrt{g^2 +h^2} } \sim - \left(\frac{h}{\sqrt{g^2 +h^2} }\right)' \sim -\left(\frac{\tilde{\eta}}{\left\vert \tilde{\eta} \right\vert}\right)_3' $$
We denote $1= r_0 > r_1 > ,r_2 > \cdots > r_N$ the zeros of $\tilde{n}_0$. Then, we claim that the signs of $\left(\frac{\tilde{\eta}}{\left\vert \tilde{\eta} \right\vert}\right)_3(r_i)$ have to alternate as $i = r_0 \cdots r_N$. Indeed, if we have that the signs are the same for $r_i$ and $r_{i+1}$, for instance $\left(\frac{\tilde{\eta}}{\left\vert \tilde{\eta} \right\vert}\right)_3(r_i)$ and $\left(\frac{\tilde{\eta}}{\left\vert \tilde{\eta} \right\vert}\right)_3(r_{i+1})$ are positive then letting $v \in \mathbb{S}^2 \cap \{ \eta(r_{i+1}) \}^{\perp} \cap \left\langle (0,0,1), \eta(r_{i+1}) \right\rangle$, the eigenfunction $\left\langle \eta,v\right\rangle$ has an isolated zero point at $\pi^{-1}(r_{i+1})$. This is a contradiction. 

Now, we deduce that $\tilde{\phi}_3^{-1}\left(\{0\}\right)\cap \mathbb{D} = \{ s_0,\cdots, s_{N} \}$ for $r_0 > s_0 > r_1 > s_2 > \cdots r_N > s_N$. Since $\tilde{\phi}_3$ is radial, symmetric with respect to the inversion $z \mapsto \frac{z}{\left\vert z \right\vert^{2}}$ and the nodal set has at most two connected components, it follows that $N=0$.

Finally, we deduce that $\frac{\eta}{\left\vert \eta \right\vert} : \mathbb{S}^2_+ \to \mathbb{S}^2 $ is an embedding. In particular, it is injective. By symmetries, $\Phi$ is an injective map, so that $\Phi$ is an embedding.
\end{proof}

\section{Proof of Theorem \ref{theoasymptotic}} \label{sectionasymptotic}

\subsection{Notations and preliminaries}
On the Euclidean sphere $(\mathbb{S}^2,h)$, we set the following competitor $g_{\eps} = e^{2\omega_{\eps}}h$, where for $y=(y_0,y_1,y_2) \in \mathbb{S}^2$ we set
$$ e^{2\omega_{\eps}(y)} = \frac{\beta_{\eps}^2 - 1}{\left( \beta_{\eps} + y_0 \right)^2} + \frac{\beta_{\eps}^2 - 1}{\left( \beta_{\eps} - y_0 \right)^2} \hskip.1cm,$$
for
$$ \beta_{\eps} = \frac{1+\eps^2}{1-\eps^2} > 1 \hskip.1cm.$$
Let's describe the surface $(\mathbb{S}^2,g_{\eps})$ in conformal coordinates.

We define $\pi_{N} : \mathbb{S}^2\setminus\{N\} \to \mathbb{R}^2$ the stereographic projection with respect to the north pole $N=(1,0,0)$ by
$$ \pi_{N}\left( y_0,y_1,y_2 \right) = \left(\frac{y_1}{1-y_0},\frac{y_2}{1-y_0} \right) $$
with inverse
$$ \pi_{N}^{-1}\left( x_1,x_2 \right) = \left(\frac{\left\vert x \right\vert^2 -1}{\left\vert x \right\vert^2 + 1}, \frac{2x_1}{\left\vert x \right\vert^2 + 1}, \frac{2x_2}{\left\vert x \right\vert^2 + 1} \right) \hskip.1cm, $$
and $\pi_{S} : \mathbb{S}^2\setminus\{S\} \to \mathbb{R}^2$ the stereographic projection with respect to the south pole $S=(-1,0,0)$ by
$$ \pi_{S}\left( y_0,y_1,y_2 \right) = \left(\frac{y_1}{1+y_0},\frac{y_2}{1+y_0} \right) \hskip.1cm,$$
with inverse
$$ \pi_{S}^{-1}\left( x_1,x_2 \right) = \left(\frac{1-\left\vert x \right\vert^2 }{\left\vert x \right\vert^2 + 1}, \frac{2x_1}{\left\vert x \right\vert^2 + 1}, \frac{2x_2}{\left\vert x \right\vert^2 + 1} \right)\hskip.1cm.  $$
Notice that $\pi_N^\star \left(e^{2u}\xi\right) = \pi_{S}^{\star} \left(e^{2u}\xi\right) =  h$ for the classical function
$$  u(x) = \ln\left( \frac{2}{1+\left\vert x \right\vert^2} \right) $$
where $x\in \mathbb{R}^2$ and $\xi$ is the Euclidean metric. Notice that $u$ satisfies the classical Liouville equation on $\mathbb{R}^2$: $\Delta u = e^{2u}$. We notice that the function $\tilde{\omega}_{\eps} : \mathbb{R}^2 \to \mathbb{R}$
$$ \tilde{\omega}_{\eps}(x) = \frac{1}{2}\ln\left( e^{2\tilde{u}_{\eps}(x)} + e^{2\tilde{v}_{\eps}(x)} \right) \hskip.1cm,$$
where
$$ e^{2\tilde{u}_{\eps}(x)} = \frac{1}{\eps^2}e^{2u\left( \frac{x}{\eps} \right)} \hbox{ and } e^{2\tilde{v}_{\eps}(x)} = \eps^2 e^{2u\left(\eps x\right)}$$
for $x\in \mathbb{R}^2$ satisfies 
$$ \pi_{S}^{\star} \left(e^{2\tilde{\omega}_{\eps}}\xi \right) = \pi_{N}^{\star} \left(e^{2\tilde{\omega}_{\eps}}\xi \right)= e^{2\omega_{\eps}}h \hskip.1cm. $$
Notice that the metrics $e^{2\tilde{u}_{\eps}(x)}\xi$ and $e^{2\tilde{v}_{\eps}(x)}\xi$ are isometric to $e^{2u}\xi$ on $\mathbb{R}^2$ as rescaled metrics of $e^{2u}\xi$. We denote
$$ e^{2\omega_{\eps}^N(x)} = e^{2\omega_{\eps}^S(x)} = \eps^2 e^{2\tilde{\omega}_{\eps}(\eps x)}  \hskip.1cm,$$
the rescaled potentials at the north and south pole. We see that the metric $e^{2\omega_{\eps}^N}\xi $ converges to $e^{2u}\xi$ in any compact set of $\mathbb{R}^2$ as $\eps\to 0$. By symmetry, $\omega_{\eps}^S = \omega_{\eps}^N$ also satisfies that $e^{2\omega_{\eps}^S}\xi $ and converges to $e^{2u}\xi$ on any compact set of $\mathbb{R}^2$. Then the surface $(\mathbb{S}^2,g_{\eps})$ geometrically represents two round spheres attached by a thin neck.
 
Now, except for the $\tilde{\omega}_{\eps}$ potentials $\omega_{\eps}^N$ and $\omega_{\eps}^S$, we denote for any function $\varphi :\mathbb{S}^2 \to \mathbb{R} $ 
$$ \widetilde{\varphi}(x) = \varphi \circ \left(\pi_N\right)^{-1}(x) \hskip.1cm,  $$
$$ \varphi^N(x) = \varphi \circ \left(\pi_N\right)^{-1}(\eps x) \hbox{ and } \varphi^S(x) = \varphi \circ \left(\pi_S\right)^{-1}(\eps x) \hskip.1cm,$$
so that an eigenvalue equation $\Delta_{h} \varphi_{\eps} = \lambda_{\eps} e^{2\omega_{\eps}} \varphi_{\eps}$ on $\mathbb{S}^2$ satisfies
$$ \Delta \widetilde{\varphi}_{\eps} = \lambda_{\eps} e^{2\tilde{\omega}_{\eps}} \widetilde{\varphi}_{\eps}\hskip.1cm, $$
$$\Delta \varphi^N_{\eps} = \lambda_{\eps} e^{2\omega^N_{\eps}} \varphi^N_{\eps} \hbox{ and } \Delta \varphi^S_{\eps} = \lambda_{\eps} e^{2\omega^S_{\eps}} \varphi^S_{\eps} \hskip.1cm.$$

Notice also that the area $A_{g_{\eps}}\left(\mathbb{S}^2\right) = 8\pi$ by a simple sum of area of standard metrics on the sphere. We set $\lambda_{1}^{\eps} = \lambda_1(g_{\eps})$ and $\lambda_2^{\eps} = \lambda_2(g_{\eps})$ the first and second eigenvalue associated to $(\mathbb{S}^2,g_{\eps})$. 

\subsection{Upper bounds for first and second eigenvalue} In this part, we aim at proving that
\begin{equation} \label{eqrayleighf1}
\lambda_1^{\eps} \leq \frac{1}{2\ln\frac{1}{\eps} } + O\left( \frac{1}{\left(\ln\frac{1}{\eps}\right)^2}\right)
\end{equation}
\begin{equation} \label{eqrayleighf2}
\lambda_2^{\eps}  \leq 2 +O(\eps^2)
\end{equation}
as $\eps\to 0$.
We test the following test functions $f_1,f_2$ in the classical variational characterization of $\lambda_1^{\eps}$ and $\lambda_2^{\eps}$:
\begin{equation} f_1\circ \pi_N^{-1}(x) = 
\begin{cases} 1 & \hbox{ on } \mathbb{D}_{\eps} \\
\frac{\ln\left(\left\vert x \right\vert\right)}{\ln \eps} & \hbox{ on } \mathbb{D}_{\frac{1}{\eps}}\setminus\mathbb{D}_{\eps} \\
-1 & \hbox{ on } \mathbb{R}^2\setminus \mathbb{D}_{\frac{1}{\eps}} \hskip.1cm.
 \end{cases}  \end{equation}
and
\begin{equation} f_2(y) = \sqrt{\beta_{\eps}^2-1} \frac{y_1}{\beta_{\eps}+y_0} \hskip.1cm, \end{equation}
noticing that $f_2(y) = e_1 . \left(\pi_N\right)^{-1} \left( \frac{\pi_N(y)}{\eps}\right)$ is a first eigenfunction associated to the metric $\frac{\beta_{\eps}^2 - 1}{\left( \beta_{\eps} + y_0 \right)^2} h$, isometric to $h$. Notice that $f_1$ and $f_2$ depend on $\eps$ but we do not indicate this dependance in the following computations. For symmetry reasons, we have that 
\begin{equation} \label{eqorthogonalityf1f2}\int_{\mathbb{S}^2} f_1 dA_{g_{\eps}} = \int_{\mathbb{S}^2} f_2 dA_{g_{\eps}} = \int_{\mathbb{S}^2} f_1 f_2 dA_{g_{\eps}} = 0 \hskip.1cm. \end{equation}
We compute
$$ \int_{\mathbb{S}^2} \left\vert \nabla f_1 \right\vert_{g_{\eps}}^2 dA_{g_{\eps}} = \int_{\mathbb{S}^2} \left\vert \nabla f_1 \right\vert_{h}^2 dA_h = \int_{\mathbb{R}^2} \left\vert \nabla f_1\circ \pi_N^{-1} \right\vert^2 dx $$
by conformal invariance of the Dirichlet energy and we easily compute that
$$ \int_{\mathbb{R}^2} \left\vert \nabla f_1\circ \pi_N^{-1} \right\vert^2 dx  =  \int_{\mathbb{D}_{\frac{1}{\eps}}\setminus \mathbb{D}_{\eps}} \left\vert \nabla f_1^N \right\vert^2 = \frac{4\pi}{\ln\frac{1}{\eps}} \hskip.1cm. $$
Moreover, we have that
\begin{eqnarray*}\int_{\mathbb{S}^2} \left( f_1 \right)^2 dA_{g_{\eps}} & = & \int_{\mathbb{R}^2} \left( \tilde{f}_1 \right)^2 e^{2\tilde{\omega}_{\eps}}dx \\
& = & 2\int_{\mathbb{D}} \left( \tilde{f}_1 \right)^2 e^{2\tilde{\omega}_{\eps}} dx \\
& = & 8\int_{\mathbb{D}}  \frac{\eps^2 \left( \tilde{f}_1 \right)^2(x)}{\left(\eps^2+\left\vert x \right\vert^2\right)^2}dx + 8\int_{\mathbb{D}}  \frac{\eps^2 \left( \tilde{f}_1 \right)^2(x)}{\left(1+\eps^2 \left\vert x \right\vert^2\right)^2} dx \\
& = & 8\int_{\mathbb{D}} \left( 1 - \frac{\ln\left\vert x \right\vert}{\ln \eps}  \right)^2  \frac{1}{\left(1+\left\vert x \right\vert^2\right)^2}dx + O(\eps^2) \\
& = & 8\pi + O\left(\frac{1}{\ln \frac{1}{\eps}}\right) \end{eqnarray*}
as $\eps\to 0$, so that \eqref{eqrayleighf1} holds.
Again, by conformal invariance, 
$$ \int_{\mathbb{S}^2} \left\vert \nabla f_2 \right\vert_{g_{\eps}}^2 dA_{g_{\eps}} = \int_{\mathbb{S}^2} \left\vert \nabla y_1 \right\vert_{h}^2 dA_h = \frac{8\pi}{3} \hskip.1cm.$$
We also compute
\begin{eqnarray*}\int_{\mathbb{S}^2} \left( f_2 \right)^2 dA_{g_{\eps}} & = & \left(\beta_{\eps}^2-1\right) \int_{\mathbb{S}^2}\frac{y_1^2}{\left(\beta_{\eps}+y_0\right)^2}  e^{2\omega_{\eps}} dA_h \\
& = & \frac{\left(\beta_{\eps}^2-1\right)}{2} \int_{\mathbb{S}^2}\frac{y_1^2+y_2^2}{\left(\beta_{\eps}+y_0\right)^2}  e^{2\omega_{\eps}} dA_h \\
& = & \frac{\left(\beta_{\eps}^2-1\right)}{2} \int_{\mathbb{S}^2_{-}}\frac{1-y_0^2}{\left(\beta_{\eps}+y_0\right)^2}  e^{2\omega_{\eps}} dA_h + O(\eps^2) \\
& = & \frac{\eps^2}{2} \int_{\mathbb{D}}  \frac{4\left\vert x \right\vert^2}{\left(1+\left\vert x \right\vert^2\right)^2} e^{2\omega_{\eps}^N}dx + O(\eps^2) \\
& = & 2 \int_{\mathbb{D}_{\frac{1}{\eps}}}  \frac{\left\vert x \right\vert^2}{\left(1+\left\vert x \right\vert^2\right)^4} dx + O(\eps^2) \\
& = & \frac{4\pi}{3} + O\left(\eps^2\right)\hskip.1cm,
 \end{eqnarray*}
as $\eps \to 0$ and we obtain \eqref{eqrayleighf2}.

\subsection{The first eigenfunction and first eigenvalue}
In this part, we aim at proving that
\begin{equation} \label{eqfirsteigenvalue} \lambda_1^{\eps} = \frac{1}{2\ln\frac{1}{\eps}} + O\left(\frac{1}{\left(\ln\frac{1}{\eps}\right)^2}\right) \end{equation}
as $\eps\to 0$. Let $\varphi_{\eps,1}$ be a first eigenfunction on $(\mathbb{S}^2,e^{2\omega_{\eps}}h)$ such that $ \int_{\mathbb{S}^2} \left(\varphi_{\eps,1}\right)^2 e^{2\omega_{\eps}}dA_h = 8\pi $. 
Up to symmetrise the eigenfunctions in case of multiplicity, we assume that 
\begin{equation} \label{symmetryfirsteigenfunction} \varphi_{\eps,1}(y_0,y_1,y_2)^2 = \varphi_{\eps,1}(-y_0,y_1,y_2)^2 \hskip.1cm. \end{equation} 
$\left(\varphi_{\eps,1}\right)^{N}$ satisfies the following equation on $\mathbb{R}^2$:
\begin{equation} \label{eqonphieps1N} \Delta \left(\varphi_{\eps,1}\right)^{N} = \lambda_1^{\eps}\frac{4}{\left(1+\left\vert x\right\vert^2\right)^2} \left(\varphi_{\eps,1}\right)^{N} + \lambda_1^{\eps}\frac{4\eps^4}{\left(1+\eps^4\left\vert x\right\vert^2\right)^2} \left(\varphi_{\eps,1}\right)^{N}\hskip.1cm,\end{equation}
with
$$ \int_{\mathbb{R}^2} \left( \frac{4}{\left(1+\left\vert x\right\vert^2\right)^2}  + \frac{4\eps^4}{\left(1+\eps^4\left\vert x\right\vert^2\right)^2}\right) \left(\left(\varphi_{\eps,1}\right)^{N}\right)^2 dx = 8\pi \hskip.1cm.$$
In particular, $\int_{\mathbb{R}^2} \frac{4}{\left(1+\left\vert x\right\vert^2\right)^2}\left(\left(\varphi_{\eps,1}\right)^{N}\right)^2 \leq 8\pi $. By standard elliptic estimates, since $\lambda_{\eps,1}\to 0$
$$ \left(\varphi_{\eps,1}\right)^{N} \to \varphi_{\star}^{N} \hbox{ in } \mathcal{C}^{2}\left(\mathbb{D}_\rho\right) $$
 as $\eps\to 0$, for any $\rho>0$ and $\Delta \varphi_{\star} = 0$ in $\mathbb{R}^2$. Since $\int_{\mathbb{R}^2} \frac{4}{\left(1+\left\vert x\right\vert^2\right)^2}\varphi_{\star}^2 \leq 8\pi$, we get that $\varphi_{\star}$ is a constant function. Looking at $\left(\varphi_{\eps,1}\right)^{S}(\eps x)$ gives the same result by symmetry and we have a constant $\varphi_{\star}^S$ at the limit. Up to take $-\varphi_{\eps,1}$, we assume that $\varphi_{\eps,1}(N)\geq 0$.
 
\medskip 
 
Now let's discuss about radial symmetries. Let $R_{\theta}$ be the rotation of axis $e_0$ and angle $\theta$. Then $M(\varphi_{\eps,1}) = \frac{1}{2\pi}\int_{0}^{2\pi} R_{\theta}^{\star} \varphi_{\eps,1} d \theta$ also satisfies the first eigenvalue equation. $M(\varphi_{\eps,1})$ only depends on $y_0$ and after stereographic projection with respect to $N$ or $S$, it is a radial function. Let's prove that $M(\varphi_{\eps,1})$ is a non-zero function. Indeed, let's prove that $\psi_{\eps} = \varphi_{\eps,1}-M(\varphi_{\eps,1})$ satisfies 
\begin{equation} \label{ineqpsiepsunifbounded} \left\|\psi_{\eps} \right\|_{\infty} \leq C\lambda_1^{\eps} \end{equation}
for a positive constant independent from $\eps$. We let $x_{\eps} \in \mathbb{D}$ be such that $\psi_{\eps}(x_{\eps}) = \left\|\psi_{\eps} \right\|_{\infty}$. We have two cases:

If $\left\vert x_{\eps} \right\vert = O\left(\eps\right)$, then $\psi_{\eps}^N$ satisfies the same equation \eqref{eqonphieps1N} as $\varphi_{\eps,1}^N$. We also know by assumption that for any $\rho>0$ we have $ \int_{\mathbb{D}_\rho} \psi_{\eps}^N = 0 $, so that by standard elliptic theory \eqref{ineqpsiepsunifbounded} holds true.

If $\eps = o\left(\left\vert x_{\eps} \right\vert\right) $, then we set $ \phi_{\eps}(z) = \widetilde{\psi}_{\eps} \left( z \left\vert x_{\eps} \right\vert \right)$ and in this case we obtain the equation
 \begin{equation} \label{eqonphieps1N2} \Delta \phi_{\eps} = \lambda_1^{\eps}\phi_{\eps} \left( \frac{4 \frac{\eps^2}{\left\vert x_{\eps} \right\vert^2}}{\left(\frac{\eps^2}{\left\vert x_{\eps} \right\vert^2}+\left\vert z\right\vert^2\right)^2}+ \frac{4\eps^2\left\vert x_{\eps} \right\vert^2}{\left(1+\eps^2\left\vert x_{\eps} \right\vert^2\left\vert z\right\vert^2\right)^2}\right)\hskip.1cm,\end{equation}
so that by standard elliptic estimates on $A = \mathbb{D}_2\setminus \mathbb{D}_{\frac{1}{2}}$, knowing that $\int_{A} \phi_{\eps} = 0$, we obtain $\left\|\psi_{\eps} \right\|_{\infty} \leq C \frac{\eps^2}{\left\vert x_{\eps} \right\vert^2}\lambda_1^{\eps}$ for some positive constant independent from $\eps$ and \eqref{ineqpsiepsunifbounded} holds true.
 
 \medskip
 
By \eqref{ineqpsiepsunifbounded}, we have that  
$$ 8\pi = \int_{\mathbb{S}^2} \left(\varphi_{\eps,1} \right)^2 e^{2\omega_{\eps}} = \int_{\mathbb{S}^2} \left(M\left(\varphi_{\eps,1}\right)\right)^2 e^{2\omega_{\eps}} + O\left(\lambda_1^{\eps}\right)$$
as $\eps\to 0$. Since $\lambda_1^{\eps} \to 0$ as $\eps\to 0$, we obtain that $M(\varphi_{\eps,1}) $ is a non-zero function. Now, up to replace $\varphi_{\eps,1}$ by $M(\varphi_{\eps,1}) $, we assume that $\tilde{\varphi}_{\eps,1}$ is a radial function.

\medskip

Since $\varphi_{\eps,1}$ is radial and is a first eigenfunction, it has to vanishes just once as a function of $y_0$ by the Courant nodal theorem. This means that with equation \eqref{symmetryfirsteigenfunction}, we have $\varphi_{\eps,1}(y_0,y_1,y_2) = - \varphi_{\eps,1}(-y_0,y_1,y_2)$ before the limit and in particular $\varphi_{\eps,1}(0,y_1,y_2) = 0$. 
By the Courant nodal theorem, $\varphi_{\eps,1}$ must be positive on $\mathbb{S}^2_+$ and negative on $\mathbb{S}^2_-$. We obtain from the equation on $\widetilde{\varphi}_{\eps,1}$ 
\begin{equation}\label{eqradialonphieps1tilde} -\frac{1}{r}\partial_r\left( r \left(\widetilde{\varphi}_{\eps,1}\right)' \right)  = \lambda_{\eps}^1 e^{2\widetilde{\omega}_{\eps}} \widetilde{\varphi}_{\eps,1} \end{equation}
that
$$ \tilde{\varphi}_{\eps,1}(r) - \tilde{\varphi}_{\eps,1}(0) = - \int_0^r \frac{1}{s} \left(\int_0^s \lambda_{\eps}^1 e^{2\tilde{\omega}_{\eps}} \tilde{\varphi}_{\eps,1}(t) t dt \right) ds \leq 0 $$
for any $r\leq 1$. This means that $\varphi_{\eps,1}$ realizes its maximum at $0$.

\medskip

Then, we obtain that for any $\rho>0$,
$$ \int_{\rho\eps}^{1} \left(\tilde{\varphi}_{\eps,1}\right)^2 e^{2\tilde{\omega}_{\eps}} dx \leq\left(\varphi_{\eps,1}(0)\right)^2\int_{\rho\eps}^{1}  e^{2\tilde{\omega}_{\eps}} dx \leq \frac{C}{\rho^2} $$
for a constant $C$ independent from $\rho$ and $\eps$. Letting $\eps\to 0$ and then $\rho\to +\infty$ on the $L^2$ formula for $\varphi_{\eps,1}$, knowing that $\varphi_{\star}^N$ and $\varphi_{\star}^S$ are constant, we get that $\left(\varphi_{\star}^N\right)^2 + \left(\varphi_{\star}^S\right)^2 = 2$. And then, since we also have 
$$ \int_{\mathbb{S}^2} \varphi_{\eps,1} e^{2\omega_{\eps}}dA_h = 0 \hskip.1cm,$$
$\varphi_{\star}^N = 1$ and $\varphi_{\star}^S = -1$. 

\medskip

Now, by the Green representation formula, 
\begin{eqnarray*} \tilde{\varphi}_{\eps,1}(0) &= & \frac{\lambda_{\eps}^1}{2\pi} \int_{\mathbb{D}} (e^{2u_{\eps}}+e^{2v_{\eps}}) \tilde{\varphi}_{\eps,1}(z)\left(\ln\frac{1}{\left\vert z\right\vert}\right) dz \\
& \leq & \lambda_{\eps}^1 \tilde{\varphi}_{\eps,1}(0) \int_{0}^{1} \frac{4\eps^2}{\left(\eps^2+r^2 \right)^2}  \left(\ln \frac{1}{r}\right) r dr + O\left(\eps^2 \right) \\
& \leq & \lambda_{\eps}^1 \tilde{\varphi}_{\eps,1}(0) \int_{0}^{\frac{1}{\eps}} \frac{4r}{\left(1+r^2 \right)^2}  \left(\ln \frac{1}{\eps r} \right)dr + O\left(\eps^2 \right) \\
& \leq & \lambda_{\eps}^1 \tilde{\varphi}_{\eps,1}(0)\left(\ln\frac{1}{\eps}\right) \int_{0}^{+\infty} \frac{4r}{\left(1+r^2 \right)^2} dr + O\left( \lambda_{1}^{\eps}\right) + O\left(\eps^2 \right) \\ \end{eqnarray*}
and knowing that $\tilde{\varphi}_{\eps,1}(0) \to 1$ as $\eps\to 0$, we obtain
$$ \lambda_1^{\eps} \geq  \frac{1}{2\ln\frac{1}{\eps}} + O\left(\frac{\lambda_{\eps}^1}{\ln\frac{1}{\eps}}\right)$$
as $\eps\to 0$. With \eqref{eqrayleighf1}, we obtain \eqref{eqfirsteigenvalue}.

\subsection{The second eigenfunction and second eigenvalue}
In this section, we aim at proving that 
\begin{equation}\label{eqprovelambda2eps} \lambda_2^\eps = 2 - 6\eps^2 + o(\eps^2) \end{equation}
as $\eps\to 0$.

We focus on the equation satisfied by $\varphi_{\eps,2}$, an eigenfunction associated to the second non-zero eigenvalue $\lambda_{\eps}^2$, satisfying 
$$ \int_{\mathbb{S}^2} \left(\varphi_{\eps,2}\right)^2 e^{2\omega_{\eps}}dA_h = \frac{8\pi}{3} \hbox{ and } \int_{\mathbb{S}^2} \varphi_{\eps,2} e^{2\omega_{\eps}}dA_h = \int_{\mathbb{S}^2} \varphi_{\eps,1}\varphi_{\eps,2} e^{2\omega_{\eps}}dA_h = 0 \hskip.1cm,$$
where $\varphi_{\eps,1}$ is the radial first eigenfunction of the previous section. Up to symmetrization again, in case of multiplicity, we assume in addition that 
$ \varphi_{\eps,2}(y_0,y_1,y_2)^2 = \varphi_{\eps,2}(-y_0,y_1,y_2)^2 $.
In fact, we must have
\begin{equation} \label{symmetrysecondeigenfunction} \varphi_{\eps,2}(y_0,y_1,y_2) = \varphi_{\eps,2}(-y_0,y_1,y_2) \hskip.1cm. \end{equation} 
Indeed, if not, we would have antisymmetry and in particular that $\varphi_{\eps,2}(0,y_1,y_2) = 0$. By orthogonality with $\varphi_{\eps,1}$ which is anti-symmetric, $\varphi_{\eps,2}$ must vanish elsewhere: $\varphi_{\eps,2}(y) = 0$ for $y_0> 0$. The nodal line containing $y$ is either a closed nodal line in $\mathbb{S}^2_+$ or contains a line in $\mathbb{S}^2_+$ joining two points in $\partial\mathbb{S}^2_+$.  Then $\varphi_{\eps,2}$ has at least two nodal domains in $\mathbb{S}^2_+$. By symmetry, $\varphi_{\eps,2}$ has at least four nodal domains, contradicting the Courant nodal theorem. We also deduce from the orthogonality with the constants that
$$ \int_{\mathbb{S}^2_+} \varphi_{\eps,2} e^{2\omega_{\eps}}dA_h = \int_{\mathbb{S}^2_-} \varphi_{\eps,2} e^{2\omega_{\eps}}dA_h = 0 \hskip.1cm. $$
Now, $\left(\varphi_{\eps,2}\right)^N$ satisfies the equation
$$ \Delta \left(\varphi_{\eps,2}\right)^{N} = \lambda_2^{\eps}\frac{4}{\left(1+\left\vert x\right\vert^2\right)^2} \left(\varphi_{\eps,2}\right)^N + \lambda_2^{\eps}\frac{4\eps^4}{\left(1+\eps^4\left\vert x\right\vert^2\right)^2} \left(\varphi_{\eps,2}\right)^{N}\hskip.1cm,$$
with
$$ \int_{\mathbb{R}^2} \left( \frac{4}{\left(1+\left\vert x\right\vert^2\right)^2}  + \frac{4\eps^4}{\left(1+\eps^4\left\vert x\right\vert^2\right)^2}\right) \left(\left(\varphi_{\eps,2}\right)^{N}\right)^2 dx = \frac{8\pi}{3} \hskip.1cm.$$
In particular, $\int_{\mathbb{R}^2} \frac{4}{\left(1+\left\vert x\right\vert^2\right)^2}\left(\left(\varphi_{\eps,2}\right)^{N}\right)^2 \leq \frac{8\pi}{3} $. $\lambda_{2}^{\eps}$ is bounded by \eqref{eqrayleighf2} and converges up to the extraction of a subsequence to $\lambda_{\star,2}$. By standard elliptic estimates, up to the extraction of a subsequence,
\begin{equation} \label{equnifconvpgieps2} \left(\varphi_{\eps,2}\right)^{N} \to \varphi_{\star,2}^{N} \hbox{ in } \mathcal{C}^{2}\left(\mathbb{D}_\rho \right) \end{equation}
for any $\rho>0$ and $\Delta \varphi_{\star,2}^N = \lambda_{\star,2} e^{2u}\varphi_{\star,2}^{N}$ on $\mathbb{R}^2$. By boundedness of the energy, $\varphi_{\star,2}\circ \pi_N$ has to be an eigenfunction associated to $\lambda_{\star,2}$ on the round sphere. Since $\lambda_{\star,2}\leq 2$ by \eqref{eqrayleighf2}, we have either $\lambda_{\star} = 0$ or $\lambda_{\star} =2$. We aim at proving that $\lambda_{\star} = 2$ and at getting an estimate on $\delta_{\eps} := \lambda_{\eps,2}-2$ as $\eps\to 0$.

\medskip

Now, let's discuss about radial symmetries. Let $R_{\theta}$ be the rotation of axis $e_0$ and angle $\theta$.
Then $M(\varphi_{\eps,2}) = \frac{1}{2\pi}\int_{0}^{2\pi} R_{\theta}^{\star} \varphi_{\eps,2} d \theta$ also satisfies the second eigenvalue equation. 
If $M(\varphi_{\eps,2}) \neq 0$, up to replace $\varphi_{\eps,2}$ by this function, we assume that $\varphi_{\eps,2}$ is radial.
Up to replace $\varphi_{\eps,2}$ by $-\varphi_{\eps,2}$, we also assume $\varphi_{\eps,2}(1,0,0) \geq 0$. We discuss this case in subsection \ref{subsectionradialcase}

If $M(\varphi_{\eps,2}) = 0$, then $\left(M(\varphi_{\eps,2})\right)(N) = \left(M(\varphi_{\eps,2})\right)(S) = 0$. Then up to a rotation, and a symmetrization, we assume that 
\begin{equation} \label{eqsymassumptionphieps2}\varphi_{\eps,2}(y_0,y_1,y_2) = - \varphi_{\eps,2}(y_0,-y_1,y_2) \hskip.1cm, \end{equation}
Up to replace $\varphi_{\eps,2}$ by $-\varphi_{\eps,2}$, we also assume $\varphi_{\eps,2}(0,1,0) \geq 0$. We discuss this case in subsection \ref{subsectionnonradialcase}

\subsubsection{The non radial case} \label{subsectionnonradialcase}
We assume that $M(\varphi_{\eps,2})=0$. Let's prove first that $\varphi_{\eps,2}$ is uniformly bounded. We let $x_{\eps} \in \mathbb{D}$ be such that $\tilde{\varphi}_{\eps,2}(x_{\eps}) = \left\|\varphi_{\eps,2} \right\|_{\infty}$. We have two cases:

If $\left\vert x_{\eps} \right\vert = O\left(\eps\right)$, then the equation on $\varphi_{\eps,2}^N$ and the assumption that for any $\rho>0$ we have $ \int_{\mathbb{D}_\rho} \varphi_{\eps,2}^N = 0 $, gives by standard elliptic theory the uniform boundedness.

If $\eps = o\left(\left\vert x_{\eps} \right\vert\right) $, then we set $ \phi_{\eps}(z) = \tilde{\varphi}_{\eps,2}\left( z \left\vert x_{\eps} \right\vert \right)$ and in this case we obtain the equation
 \begin{equation} \label{eqonphieps2N} \Delta \phi_{\eps} = \lambda_2^{\eps}\phi_{\eps} \left( \frac{4 \frac{\eps^2}{\left\vert x_{\eps} \right\vert^2}}{\left(\frac{\eps^2}{\left\vert x_{\eps} \right\vert^2}+\left\vert z\right\vert^2\right)^2}+ \frac{4\eps^2\left\vert x_{\eps} \right\vert^2}{\left(1+\eps^2\left\vert x_{\eps} \right\vert^2\left\vert z\right\vert^2\right)^2}\right)\hskip.1cm,\end{equation}
 so that by standard elliptic estimates on $A = \mathbb{D}_2\setminus \mathbb{D}_{\frac{1}{2}}$, knowing that $\int_{A} \phi_{\eps} = 0$, we obtain $\left\|\tilde{\varphi}_{\eps,2} \right\|_{\infty} \leq C \frac{\eps^2}{\left\vert x_{\eps} \right\vert^2}\lambda_2^{\eps}$ for some positive constant independent from $\eps$ and $\tilde{\varphi}_{\eps,2}$ is uniformly bounded.
 
 \medskip
 
Then, since $\tilde{\varphi}_{\eps,2}$ is uniformly bounded, we have that for any $\rho>0$, 
$$ \int_{\rho\eps}^{1} \left(\tilde{\varphi}_{\eps,2}\right)^2 e^{2\tilde{\omega}_{\eps}} dx \leq\frac{C}{\rho^2} $$
for a constant $C$ independent from $\rho$ and $\eps$. Letting $\eps\to 0$ and then $\rho\to +\infty$ on the $L^2$ formula for $\varphi_{\eps,2}$, we get that
$$ \int_{\mathbb{R}^2} \left(\varphi_{\star,2}^N\right)^2 e^{2u}dx =  \frac{4\pi}{3} \hbox{ and } \int_{\mathbb{R}^2}\varphi_{\star,2}^N e^{2u}dx  = 0 \hskip.1cm.$$
Then, by orthogonality with the constant functions, we must have that $\lambda_{\star,2} = 2$, and the inverse stereographic projection of $\left(\varphi_{\star,2}\right)^N $ has to be a first eigenfunction on the sphere. We have the following linear combination:
$$ \varphi_{\star,2}^N =  \sum_{i=0}^3 a_{\star,i} Y_i \hskip.1cm, $$
where $Y_i(x) = \pi_N^{-1}\left(x\right).e_i$ with $\sum_{i=0}^3 \left(a_{\star,i}\right)^2 = 1$. By the symmetry before the limit \eqref{eqsymassumptionphieps2}, we have $\varphi_{\star,2}^N(y_0,0,y_2) = 0$ and $\varphi_{\star,2}^N(0,1,0)>0$. We obtain that $a_{\star,1} = 1$ and that $a_{\star,0} = a_{\star,2} = 0$.

\medskip

Now, we aim at estimating $\delta_{\eps} =  \lambda_{\eps}^2 -2 $.

\medskip

We define $R_{\eps} : \mathbb{S}^2 \to \mathbb{R}$ such that its inverse stereographic projection $\tilde{R}_{\eps} = R_{\eps} \circ \left(\pi_N\right)^{-1}$ satisfies
\begin{equation}\tilde{R}_{\eps}(x) = \varphi_{\eps,2}(x) - \eta_{\eps}\left(Y_1\left(\frac{x}{\eps}\right) + Y_1\left(x\eps\right)\right) \hskip.1cm. \end{equation}
and such that $R_{\eps} \circ \left(\pi_S\right)^{-1}$ satisfies the same formula by symmetry, where $\eta_{\eps}$ is defined such that $\nabla R_{\eps}(S) =0$ and $ \nabla R_{\eps}(N) =0$. Such a $\eta_{\eps}$ exists because by the symmetry assumption \eqref{eqsymassumptionphieps2}, $\partial_{x_2} \varphi_{\eps,2}^N(0) = \partial_{x_2} Y^1(0) = 0 $ and $\eta_{\eps}$ can be defined as
$$ \eta_{\eps} = \frac{\partial_{x_1}\tilde{\varphi}_{\eps,2}(0)}{ \left(\frac{1}{\eps}+\eps\right) \partial_{x_1}Y^1(0) } = \frac{\partial_{x_1}\varphi^N_{\eps,2}(0)}{2\left(1+\eps^2\right)} \hskip.1cm.$$
Notice that $\eta_{\eps} = 1+o(1)$ as $\eps\to 0$. We obtain the following equation of $R_{\eps}$ on $\mathbb{S}^2$:
\begin{equation}
\label{eqonRepsnonradialsphere}
\Delta_h R_{\eps} - 2 e^{2\omega_{\eps}}R_{\eps} = \left(\lambda_{\eps}^2-2\right)e^{2 \omega_{\eps}}\varphi_{\eps,2} + 2\eta_{\eps}\left( e^{2u_{\eps}} Y_{1,\eps}^N  + e^{2v_{\eps}} Y_{1,\eps}^S\right)
\end{equation}
then integrating against $R_{\eps}$ in $\mathbb{S}^2$,
\begin{equation}\label{eqW12estonRepsnonradial}
\begin{split}
\int_{\mathbb{S}^2} \left\vert \nabla R_{\eps}\right\vert_h^2 =&  2 \int_{\mathbb{S}^2} \left(R_{\eps}\right)^2 e^{2u_{\eps}} + \left(\lambda_{\eps}^2-2\right) \int_{\mathbb{S}^2}e^{2\omega_{\eps}} \varphi_{\eps,2}R_{\eps} + 2\eta_{\eps} \int_{\mathbb{S}^2} \left(e^{2u_{\eps}} Y_{1,\eps}^N  + e^{2v_{\eps}} Y_{1,\eps}^S\right) R_{\eps} \\
\leq & \left\| R_{\eps} \right\|_{\infty} \left( \left\| R_{\eps} \right\|_{\infty} +2\pi \left\| \varphi_{\eps,2} \right\|_{\infty} \left\vert \lambda_{\eps}^2 -2 \right\vert + O\left(\eps^2\right) \right) \\
\leq & C \left\| R_{\eps} \right\|_{\infty} \left(\left\| R_{\eps} \right\|_{\infty} + \delta_{\eps} + \eps^2  \right) \hskip.1cm,
\end{split}
\end{equation}
where we computed $I:= \int_{\mathbb{S}^2} \left(e^{2u_{\eps}} Y_{1,\eps}^N  + e^{2v_{\eps}} Y_{1,\eps}^S\right) R_{\eps}$ as
\begin{equation*} 
\begin{split}
I = & 2 \int_{\mathbb{S}^2_-} \left(e^{2u_{\eps}} Y_{1,\eps}^N  + e^{2v_{\eps}} Y_{1,\eps}^S\right) R_{\eps} \\
= & \int_{\mathbb{D}} \left(\frac{4\eps^2}{\left(1+ \eps^2 \left\vert x \right\vert^2  \right)^2} \frac{2x_1 \eps}{\eps^2 + \left\vert x \right\vert^2 } + \frac{4\eps^2}{\left(\eps^2+ \left\vert x \right\vert^2  \right)^2} \frac{2x_1 \eps}{1 + \eps^2\left\vert x \right\vert^2 } \right) \tilde{R}_{\eps} \\
\leq & 2\pi \left\| R_{\eps} \right\|_{\infty} \int_0^1 \left( \frac{4 \eps^2}{\left(1+ \eps^2 r^2  \right)^2} \frac{2 r \eps}{\eps^2 + r^2 } + \frac{4\eps^2}{\left(\eps^2+ r^2  \right)^2} \frac{2r \eps}{1 + \eps^2r^2 }  \right)r dr \\
= & 2\pi \left\| R_{\eps} \right\|_{\infty} \eps^2 \int_0^{\frac{1}{\eps}} \left( \frac{4 u \eps}{\left(1+ \eps^4 u^2  \right)^2} \frac{2 u }{1 + u^2 } + \frac{4 u}{\left(1+ u^2  \right)^2} \frac{2u}{1 + \eps^4 u^2 }  \right) du \\
\leq & 2\pi \left\| R_{\eps} \right\|_{\infty} \eps^2 \left( 8 + 2\pi \right) = O\left( \left\| R_{\eps} \right\|_{\infty} \eps^2 \right)
\end{split}
\end{equation*}
as $\eps\to 0$. Letting 
\begin{equation}\label{defalphaepsnonradial} \alpha_{\eps} = \left\| R_{\eps} \right\|_{L^\infty\left(\mathbb{D}\right)} + \left\vert\delta_{\eps}\right\vert + \eps^2 \hskip.1cm,\end{equation}
we get that 
\begin{equation} \label{eqW12estonRepsnonradial2}
\int_{\mathbb{D}} \left\vert \nabla R_{\eps}\right\vert^2 \leq O\left( \alpha_{\eps}^2 \right)
\end{equation}
as $\eps\to 0$. Letting $R_{\eps}^N(x) = R_{\eps}(\eps x)$, we get from \eqref{eqonRepsnonradialsphere} that
\begin{equation}
\label{eqonRNepsnonradial}
\begin{split}
\Delta R_{\eps}^N - 2 e^{2u} R_{\eps}^N = & \left(\lambda_{\eps}^2-2\right)e^{2 u}\tilde{\varphi}_{\eps,2}^N  + 2 e^{2u} \eta_{\eps} Y_1(\eps^2 x) \\
& + \frac{4\eps^4}{\left(1+\eps^4 \left\vert x \right\vert^2\right)^2}  \left( \lambda_{\eps}^2\tilde{\varphi}^N_{\eps,2}- 2 \eta_{\eps}Y_1(\eps^2 x) \right)
\hskip.1cm,
\end{split}
\end{equation}
where $Y_1(\eps^2 x) = \frac{2 \eps^2 x_1}{1+\eps^4 \left\vert x \right\vert^2}$.
Dividing this equation by $\alpha_{\eps}$, by standard elliptic theory, we obtain up to the extraction of a subsequence that 
\begin{equation} \label{defRNstarnonradial} \frac{R_{\eps}^N}{\alpha_{\eps}} \to R_{\star}^{N} \hbox{ in } \mathcal{C}^{2}\left(\mathbb{D}_\rho\right) \end{equation}
as $\eps\to 0$ for any $\rho>0$. We also have that up to the extraction of a subsequence
$$ \frac{\delta_{\eps}}{\alpha_{\eps}} \to \delta_{\star} \text{ and } \frac{\eps^2}{\alpha_{\eps}}\to e_{\star} \text{ as } \eps\to 0 \hskip.1cm.$$ 
Passing to the limit in \eqref{eqonRNepsnonradial} divided by $\alpha_{\eps}$, $R_{\star}^{N} $ satisfies
\begin{equation} \label{eqonRNstarnonradial} 
\Delta R_{\star}^N - 2 e^{2u} R_{\star}^N = \delta_{\star} e^{2 u} Y^1 + 4  e^{2u} x_1 e_{\star} \hskip.1cm.
 \end{equation}
We set $R_{\star} := R_{\star}^N \circ \pi_N$ in $\mathbb{S}^2 \setminus \{N\}$. We obtain
\begin{equation} \label{eqonRNstarnonradialsphere} 
\Delta_h R_{\star} - 2  R_{\star} = \left(\delta_{\star} +  e_{\star} \frac{4}{1-y_0}\right) y_1
 \end{equation}
in $\mathbb{S}^2\setminus\{N\}$. Notice that $\frac{y_1}{1-y_0}$ is a bounded function. By \eqref{eqW12estonRepsnonradial2}, we have that $R_{\star} \in W^{1,2}(\mathbb{S}^2)$ and we can extend $R_{\star}$ in $\mathbb{S}^2$ so that the equation holds in $\mathbb{S}^2$. Integrating this equation against $y_1$, since $\Delta_h y_1 - 2 y_1 = 0$, we obtain that 
$$ \delta_{\star} = -   \frac{4\int_{\mathbb{S}^2} \frac{\left(y_1\right)^2}{1-y_0}dA_h}{\int_{\mathbb{S}^2} \left(y_1\right)^2dA_h} e_{\star} = -  \frac{8\pi}{\frac{4\pi}{3}} e_{\star} = -6 e_{\star} \hskip.1cm.$$
Notice now that if $e_{\star} \neq 0$, then
\begin{equation}\label{eqfinalargestarnonzero} \frac{\delta_{\eps}}{\eps^2} = \frac{\frac{\delta_{\eps}}{\alpha_{\eps}}}{\frac{\eps^2}{\alpha_{\eps}}} \to \frac{\delta_{\star}}{e_{\star}} = - 6 \hskip.1cm,\end{equation}
And \eqref{eqprovelambda2eps} would be proved in the non radial case.

\medskip

From now to the end of subsection \ref{subsectionnonradialcase}, we assume by contradiction that $e_{\star} = 0$. This implies that $\delta_{\star} = 0$. We obtain by \eqref{defalphaepsnonradial} and by definition of $\delta_{\star}$ and $e_{\star}$
\begin{equation} \label{eqestimateondeltaepsnonradial} \delta_{\eps} = o\left( \left\| R_{\eps} \right\|_{\infty}  \right) \text{ and } \eps^2 = o\left( \left\| R_{\eps} \right\|_{\infty}  \right)
\end{equation}
as $\eps\to 0$. Moreover, \eqref{eqonRNstarnonradialsphere} becomes $\Delta_h R_{\star} -2R_{\star} = 0$. This means that $R_{\star}$ is a first eigenfunction in $\mathbb{S}^2$. By assumptions $R_{\eps}(S)=R_{\eps}(N)=0$, $\nabla R_{\eps}(S)=0$ and $\nabla R_{\eps}(N)=0$, we obtain that $R_{\star}(S) = 0$ and $\nabla R_{\star}(S)=0$. Therefore, we must have $R_{\star}=0$. We obtain that for any $\rho>0$
\begin{equation} \label{eqestimateonRepsRepsnonradial} R_{\eps}(\rho \eps ) = o\left( \left\| R_{\eps} \right\|_{\infty} \right) 
\end{equation}
as $\eps\to 0$. 

Let $x_{\eps}\in \mathbb{D}$ be such that $\left\| R_{\eps} \right\|_{\infty} = \left\vert R_{\eps}(x_{\eps}) \right\vert$. We set $r_{\eps} =\left\vert x_{\eps} \right\vert$. We obtain from \eqref{eqestimateonRepsRepsnonradial} that $\eps = o\left(r_{\eps}\right)$. We set $\psi_{\eps}(z) = R_{\eps}(r_{\eps}z)$ and $\phi_{\eps}(z) = \tilde{\varphi}_{\eps,2}(r_{\eps}z)$ and we get from \eqref{eqonRepsnonradialsphere}
 \begin{equation} \label{eqonRphieps2} 
 \begin{split}\Delta \psi_{\eps} = & \left( 2 \psi_{\eps} + \left(\lambda_2^{\eps}-2\right) \phi_{\eps} + 2\eta_{\eps}Y_1(r_{\eps}\eps z) \right)  \frac{4 \frac{\eps^2}{r_{\eps}^2}}{\left(\frac{\eps^2}{r_{\eps}^2} +\left\vert z\right\vert^2\right)^2} \\
 &+ \left(\lambda_2^{\eps}\phi_{\eps}-2\eta_{\eps}Y_1(r_{\eps}\eps z)\right) \frac{4\eps^2r_{\eps}^2}{\left(1+\eps^2r_{\eps}^2\left\vert z\right\vert^2\right)^2} \hskip.1cm,
 \end{split}
\end{equation}
 so that dividing by $\alpha_{\eps}$, we obtain that for any $\rho>1$, 
 $$\left\|  \Delta \psi_{\eps}\right\|_{L^\infty\left(\mathbb{D}_{\rho}\setminus\mathbb{D}_{\frac{1}{\rho}}\right)} = O\left( \left(\left\| \psi_{\eps}\right\|_{\infty}  + \left\vert \delta_{\eps} \right\vert\right) \frac{\eps^2}{r_{\eps}^2}+ \frac{\eps^3}{r_{\eps}}\right) = O\left( \left\| \psi_{\eps}\right\|_{\infty}\frac{\eps}{r_{\eps}}\right) \hskip.1cm.$$
By standard elliptic theory, we obtain up to the extraction of a subsequence that 
$$\frac{\psi_{\eps}}{\alpha_{\eps}} \to \psi_{\star} \text{ in } \mathcal{C}^2\left(\mathbb{D}_\rho\setminus \mathbb{D}_{\frac{1}{\rho}}\right)$$
as $\eps\to 0$ for any $\rho>0$.  We also define up to the extraction of a subsequence
$$ \frac{x_{\eps}}{r_{\eps}} \to x_{\star}\in \mathbb{S}^1 \hbox{ as } \eps\to 0 \hskip.1cm. $$
Passing to the limit in \eqref{eqonRphieps2} divided by $\alpha_{\eps}$, $\psi_{\star}$ is a harmonic function on $\mathbb{R}^2\setminus\{0\}$. Since $\psi_{\star}$ is bounded by $1$, $\psi_{\star}$ is a constant function. But since the mean value of $\psi_{\star}$ is equal to $0$ on any circle centered at $0$, we obtain that $\psi_{\star}(x_{\star}) = 0$. Therefore $\left\| R_{\eps} \right\|_{L^\infty\left(\mathbb{D}\right)} = o\left(\alpha_{\eps}\right)$ and it is a contradiction with \eqref{defalphaepsnonradial} and \eqref{eqestimateondeltaepsnonradial}.

\medskip

Therefore, $e_{\star}\neq 0$ and we proved \eqref{eqprovelambda2eps} thanks to \eqref{eqfinalargestarnonzero} in the non-radial case.

\subsubsection{The radial case} \label{subsectionradialcase}
We aim at proving that $M\left(\widetilde{\varphi}_{\eps,2}\right) = 0$ so that the second eigenfunctions behave like in section \eqref{subsectionnonradialcase}.

We assume by contradiction that $M\left(\widetilde{\varphi}_{\eps,2}\right)$ is a non zero function. We aim at proving that $\lambda_{\eps}^2-2 = o(\eps^2)$ as $\eps \to 0$ which would give a contradiction. As already discussed, up to take $M\left(\widetilde{\varphi}_{\eps,2}\right)$ instead of $\widetilde{\varphi}_{\eps,2}$, we assume that $\widetilde{\varphi}_{\eps,2}$ is a radial function. We also assume that $\varphi_{\eps,2}(1,0,0) \geq 0$. Since $\varphi_{\eps,2}$ has to vanish somewhere and since it is symmetric, see \eqref{symmetrysecondeigenfunction}, we must have the existence of $0\leq r_{\eps}\leq 1$ such that $\widetilde{\varphi}_{\eps,2}(r_{\eps})=0$. Moreover, by symmetry, the Courant nodal theorem, and assumption $\widetilde{\varphi}_{\eps,2}(0)\geq 0$, we have that $ \widetilde{\varphi}_{\eps,2}$ is positive on $ [0,r_{\eps}[$ and negative on $ ]r_{\eps},1]$. By symmetry again, we must have $\left(\widetilde{\varphi}_{\eps,2}\right)'(1) = 0$.

Therefore, we integrate the eigenvalue equation on the radial function $\widetilde{\varphi}_{\eps,2}$
\begin{equation}\label{eqradialonphieps2tilde} -\frac{1}{r}\partial_r\left( r \left(\widetilde{\varphi}_{\eps,2}\right)' \right)  = \lambda_{\eps}^2 e^{2\widetilde{\omega}_{\eps}} \widetilde{\varphi}_{\eps,2} \end{equation}
and we obtain for $0< s< 1$
\begin{equation} \label{eqintradialonphieps2tilde} 
\left(\widetilde{\varphi}_{\eps,2}\right)'(s) = - \frac{1}{s}\int_0^s \lambda_{\eps}^2 e^{2\widetilde{\omega}_{\eps}} \widetilde{\varphi}_{\eps,2}(t) t dt =  \frac{1}{s}\int_s^1 \lambda_{\eps}^2 e^{2\widetilde{\omega}_{\eps}} \widetilde{\varphi}_{\eps,2}(t) t dt  < 0 \hskip.1cm, \end{equation}
so that $\widetilde{\varphi}_{\eps,2}$ realizes its maximum for $r=0$ and its minimum for $r=1$.

\medskip

Now, let's prove that $\widetilde{\varphi}_{\eps,2}$ is uniformly bounded. Since $\widetilde{\varphi}_{\eps,2}(0) = \varphi^N_{\star,2}(0) + o(1)$, it suffices to prove that $\widetilde{\varphi}_{\eps,2}(1)$ is uniformly bounded. Integrating again \eqref{eqintradialonphieps2tilde}, we have that
\begin{equation*}
\begin{split}
\widetilde{\varphi}_{\eps,2}(r)  - \widetilde{\varphi}_{\eps,2}(1) = & \int_{r}^1 \frac{1}{s}\left(\int_s^1 \lambda_{\eps}^2 e^{2\widetilde{\omega}_{\eps}} \left(-\widetilde{\varphi}_{\eps,2}(t)\right) t dt\right)ds \\
\leq &  - \widetilde{\varphi}_{\eps,2}(1)  \lambda_{\eps}^2  \int_{r}^1 \frac{1}{s}\left(\int_s^1  e^{2\widetilde{\omega}_{\eps}} t dt\right)ds \\
\leq & -\widetilde{\varphi}_{\eps,2}(1) \lambda_{\eps}^2 \left(  \ln\left(1 +\frac{\eps^2}{r^2}\right) + \eps^2 \right)
\end{split}
 \end{equation*}
so that choosing $r = \rho\eps$ for any fixed positive constant $\rho>0$, we have by \eqref{equnifconvpgieps2} that $\widetilde{\varphi}_{\eps,2}(\rho\eps) = \varphi^N_{\star,2}(R)+o(1) $ as $\eps\to 0$ and
\begin{equation*}
- \widetilde{\varphi}_{\eps,2}(1)  \left(1 - \lambda_{\eps}^2 \left(  \ln\left(1 +\frac{1}{\rho^2}\right) + \eps^2 \right) \right) \leq  -\varphi^N_{\star,2}(\rho)+o(1) 
 \end{equation*}
 as $\eps \to 0$. Since we have \eqref{eqrayleighf2}, choosing $\rho$ such that $2  \ln\left(1 +\frac{1}{\rho^2}\right) = \frac{1}{2}$, we obtain that $ \widetilde{\varphi}_{\eps,2}(1)$ is uniformly bounded.

\medskip

Then, since $\tilde{\varphi}_{\eps,2}$ is uniformly bounded, we have that for any $\rho>0$, 
$$ \int_{\rho\eps}^{1} \left(\tilde{\varphi}_{\eps,2}\right)^2 e^{2\tilde{\omega}_{\eps}} dx \leq\frac{C}{\rho^2} $$
for a constant $C$ independent from $\rho$ and $\eps$. Letting $\eps\to 0$ and then $\rho\to +\infty$ on the $L^2$ formula for $\varphi_{\eps,2}$, we get that
$$ \int_{\mathbb{R}^2} \left(\varphi_{\star,2}^N\right)^2 e^{2u}dx =  \frac{4\pi}{3} \hbox{ and } \int_{\mathbb{R}^2}\varphi_{\star,2}^N e^{2u}dx  = 0 \hskip.1cm.$$
Then, by orthogonality with the constant functions, we must have that $\lambda_{\star,2} = 2$, and the inverse stereographic projection of $\left(\varphi_{\star,2}\right)^N $ has to be a first eigenfunction on the sphere. We have the following linear combination:
$$ \varphi_{\star,2}^N =  \sum_{i=0}^3 a_{\star,i} Y_i \hskip.1cm, $$
where $Y_i(x) = \pi_N^{-1}\left(x\right).e_i$ with $\sum_{i=0}^3 \left(a_{\star,i}\right)^2 = 1$. Since $\varphi_{\star,2}^N$ is radial and $\varphi_{\star,2}^N(0)>0$, we obtain that $a_{\star,0} = 1$ and that $a_{\star,1} = a_{\star,2} = 0$.

\medskip

Now, we aim at estimating $\delta_{\eps} = \left\vert \lambda_{\eps}^2 -2 \right\vert$.

\medskip

We set $R_{\eps}(x) = \tilde{\varphi}_{\eps,2}(x) - \tilde{\varphi}_{\eps,2}(0) Y^0\left(\frac{x}{\eps}\right)$ so that $R_{\eps}$ is a radial function satisfying $R_{\eps}(0)=0$ satisfying the following equation
\begin{equation}
\label{eqonRepsradial}
\Delta R_{\eps} - 2 e^{2u_{\eps}}R_{\eps} = \left(\lambda_{\eps}^2-2\right)e^{2 u_{\eps}}\tilde{\varphi}_{\eps,2} + \lambda_{\eps}^2 e^{2v_{\eps}} \tilde{\varphi}_{\eps,2} 
\end{equation}
in $\mathbb{D}$. Integrating this equation against $R_{\eps}$, we obtain 
\begin{equation}\label{eqW12estonRepsradial}
\begin{split}
\int_{\mathbb{D}} \left\vert \nabla R_{\eps}\right\vert^2 =& \int_{\mathbb{S}^1} R_{\eps}\partial_{\nu} R_{\eps} + \left(\lambda_{\eps}^2-2\right) \int_{\mathbb{D}}e^{2u_{\eps}} \tilde{\varphi}_{\eps,2}R_{\eps} + \lambda_{\eps}^2 \int_{\mathbb{D}} e^{2v_{\eps}} \tilde{\varphi}_{\eps,2} R_{\eps} \\
\leq & \left\| \varphi_{\eps,2} \right\|_{\infty} \left\| R_{\eps} \right\|_{\infty} \left( \int_{\mathbb{S}^1} \left\vert \left(\partial_{\nu} Y^0\left(\frac{x}{\eps}\right)\right) \right\vert + \left\| R_{\eps} \right\|_{\infty} +2\pi \left\vert \lambda_{\eps}^2 -2 \right\vert + O\left(\eps^2\right) \right) \\
\leq & C \left\| R_{\eps} \right\|_{\infty} \left(\left\| R_{\eps} \right\|_{\infty} + \delta_{\eps} + \eps^2  \right)
\end{split}
\end{equation}
for a positive constant $C$ independent from $\eps$. Letting 
\begin{equation}\label{defalphaepsradial} \alpha_{\eps} = \left\| R_{\eps} \right\|_{\infty} + \delta_{\eps} + \eps^2 \hskip.1cm,\end{equation}
we get that 
\begin{equation} \label{eqW12estonRepsradial2}
\int_{\mathbb{D}} \left\vert \nabla R_{\eps}\right\vert^2 \leq O\left( \alpha_{\eps}^2 \right)
\end{equation}
as $\eps\to 0$. Letting $R_{\eps}^N(x) = R_{\eps}(\eps x)$, we get from \eqref{eqonRepsradial} that
\begin{equation}
\label{eqonRNepsradial}
\Delta R_{\eps}^N - 2 e^{2u} R_{\eps}^N = \left(\lambda_{\eps}^2-2\right)e^{2 u}\tilde{\varphi}_{\eps,2}^N + \lambda_{\eps}^2 \frac{4\eps^4}{\left(1+\eps^4 \left\vert x \right\vert^2\right)^2} \tilde{\varphi}^N_{\eps,2} \hskip.1cm,
\end{equation}
So that dividing this equation by $\alpha_{\eps}$, by standard elliptic theory, we obtain up to the extraction of a subsequence that 
\begin{equation} \label{defRNstarradial} \frac{R_{\eps}^N}{\alpha_{\eps}} \to R_{\star}^{N} \hbox{ in } \mathcal{C}^{2}\left(\mathbb{D}_R\right) \end{equation}
as $\eps\to 0$ for any $\rho>0$. We also have that $\frac{\delta_{\eps}}{\alpha_{\eps}} \to \delta_{\star}$ as $\eps \to 0$ and $\frac{\eps^4}{\alpha_{\eps}}\to 0$ as $\eps\to 0$. From \eqref{eqonRNepsradial}, $R_{\star}^{N} $ satisfies
\begin{equation} \label{eqonRNstarradial} 
\Delta R_{\star}^N - 2 e^{2u} R_{\star}^N = \delta_{\star} e^{2 u} Y^0 \hskip.1cm.
 \end{equation}
We set $R_{\star} := R_{\star}^N \circ \pi_N$ in $\mathbb{S}^2 \setminus \{N\}$. We obtain
\begin{equation} \label{eqonRNstarradialsphere} 
\Delta_h R_{\star} - 2  R_{\star} = \delta_{\star} y_0 
 \end{equation}
in $\mathbb{S}^2\setminus\{N\}$. By \eqref{eqW12estonRepsradial2}, we have that $R_{\star} \in W^{1,2}(\mathbb{S}^2)$ and we can extend $R_{\star}$ in $\mathbb{S}^2$ so that the equation holds in $\mathbb{S}^2$. Integrating \eqref{eqonRNstarradialsphere} against $y_0$ gives $\delta_{\star} = 0$. By definition of $\delta_{\star}$, we obtain that $\delta_{\eps} = o\left( \alpha_{\eps} \right) $ as $\eps\to 0$, and we obtain by \eqref{defalphaepsradial}
\begin{equation} \label{eqestimateondeltaeps} \delta_{\eps} = o\left( \left\| R_{\eps} \right\|_{\infty} +  \eps^2 \right) 
\end{equation}
as $\eps\to 0$.

Moreover, we obtain that $\Delta_h R_{\star} = 2  R_{\star}$. This means that $R_{\star}$ is a first eigenfunction in $\mathbb{S}^2$. Since $R_{\star}$ only depends on $y_0$ and $R_{\star}(S)=0$, we obtain that $R_{\star}=0$. We obtain that for any $\rho>0$
\begin{equation} \label{eqestimateonRepsReps} R_{\eps}(\rho \eps ) = o\left( \left\| R_{\eps} \right\|_{\infty} +  \eps^2 \right) 
\end{equation}
as $\eps\to 0$.

\medskip

Let's prove now that 
\begin{equation} \label{eqestonnormReps}\left\| R_{\eps} \right\|_{\infty} = O\left(\eps^2\right) \end{equation}
as $\eps\to 0$. Let $0 \leq r_{\eps} \leq 1$ be such that $\left\| R_{\eps} \right\|_{\infty} = \left\vert R_{\eps}(r_{\eps}) \right\vert$. By \eqref{eqestimateondeltaeps}, if $r_{\eps} = O(\eps)$, we easily deduce \eqref{eqestonnormReps}. We then assume that $\eps = o(r_{\eps})$. By integration on the equation \eqref{eqonRepsradial} satisfied by the radial function $R_{\eps}$, we have that for $r\leq 1$,
\begin{equation}
\label{eqradialestonReps}
\begin{split}
R_{\eps}(r) - R_{\eps}(1) =  & -\int_{1}^r \frac{1}{s}\left(\int_{1}^s \Delta R_{\eps} t dt\right) ds \\
= &  -2\int_{1}^r \frac{1}{s}\left(\int_{1}^s R_{\eps} e^{2u_{\eps}} t dt\right) ds - \left(\lambda_\eps^2 - 2\right)\int_{1}^r\frac{1}{s}\left(\int_{1}^s e^{2\tilde{\omega}_{\eps}}\tilde{\varphi}_{\eps,2} t dt\right)ds \\
& - 2 \int_{1}^r \frac{1}{s} \left(\int_{1}^s e^{2v_{\eps}}\tilde{\varphi}_{\eps,2} t dt \right)ds  + (1-r)\left(R_{\eps}\right)'(1) 
\end{split}
\end{equation}
so that since $R_{\eps}'(1) = O\left(\eps^2\right)$, that
$\int_{1}^r \frac{1}{s}\left(\int_{1}^s e^{2u_{\eps}} t dt\right) ds = \ln\left(1+\frac{\eps^2}{r^2}\right)$
and that  
$$\varphi_{\eps}(r) - \varphi_{\eps}(1) = -  \int_{1}^r\frac{1}{s}\left(\int_{1}^s e^{2\tilde{\omega}_{\eps}}\tilde{\varphi}_{\eps,2} t dt\right)ds \hskip.1cm,$$
we have since $\tilde{\varphi}_{\eps,2}$ is uniformly bounded and by definition of $\delta_{\eps}$ that
\begin{equation} \label{eqradialestonReps2}
\left\vert R_{\eps}(r) - R_{\eps}(1)  \right\vert \leq 2\left\| R_{\eps} \right\|_{\infty}\ln\left(1+\frac{\eps^2}{r^2}\right) + O\left( \delta_{\eps} + \eps^2 \right) \hskip.1cm.
\end{equation}
In particular for $r=r_{\eps}$, since we have \eqref{eqestimateondeltaeps}, that $R_{\eps}(r_{\eps}) = \left\| R_{\eps} \right\|_{\infty}$ and that $\eps = o(r_{\eps})$, we obtain that
\begin{equation} \label{eqradialestonReps3}
\left\| R_{\eps} \right\|_{\infty} \leq \left\vert R_{\eps}(1)  \right\vert + O\left( \eps^2 \right) \hskip.1cm.
\end{equation}
Applying again \eqref{eqradialestonReps2} for $r= \rho\eps$ for a fixed $\rho>0$ such that $2\ln\left(1+\frac{1}{\rho^2}\right) = \frac{1}{2}$, we obtain that
\begin{equation} \label{eqradialestonReps4}
\left\vert R_{\eps}(1) \right\vert \leq \left\vert R_{\eps}(\rho\eps)  \right\vert + O\left( \eps^2 \right)  \hskip.1cm.
\end{equation}
as $\eps\to 0$. By \eqref{eqestimateonRepsReps}, we easily obtain the expected result \eqref{eqestonnormReps}.

\medskip

Thanks to \eqref{eqestonnormReps} and \eqref{eqestimateondeltaeps}, we then obtain that $\delta_{\eps} = o(\eps^2)$ in this case. This concludes subsection \eqref{subsectionradialcase}.

\bibliographystyle{alpha}
\bibliography{mybibfile}

\end{document}